\documentclass{amsart}

\usepackage{amsmath}
\usepackage{amsthm}
\usepackage{latexsym}
\usepackage{braket}
\usepackage{graphicx}
\usepackage{stmaryrd}
\usepackage{xcolor}
\usepackage[caption=false,font=footnotesize]{subfig}

\newtheorem{proposition}{Proposition}
\newtheorem{algorithm}{Algorithm}
\theoremstyle{definition}
\newtheorem{remark}{Remark}
\newtheorem*{notation}{Notation}

\definecolor{blue}{RGB}{16,78,139}

\begin{document}

\title[Computing distances and geodesics between curves in a manifold]{Computing distances and geodesics between manifold-valued curves in the SRV framework}

\author{Alice Le Brigant}
\address{Institut Math\'ematique de Bordeaux, UMR 5251, 
Universit\'e de Bordeaux and CNRS, France}
\email{alice.lebrigant@math.u-bordeaux.fr}

\begin{abstract}

This paper focuses on the study of open curves in a Riemannian manifold $M$, and proposes a reparametrization invariant metric on the space of such paths. We use the square root velocity function (SRVF) introduced by Srivastava et al. in \cite{sri} to define a Riemannian metric on the space of immersions $\mathcal{M}=\text{Imm}([0,1],M)$ by pullback of a natural metric on the tangent bundle $\text{T}\mathcal{M}$. This induces a first-order Sobolev metric on $\mathcal{M}$ and leads to a distance which takes into account the distance between the origins in $M$ and the $L^2$-distance between the SRV representations of the curves. The geodesic equations for this metric are given and exploited to define an exponential map on $\mathcal M$. The optimal deformation of one curve into another can then be constructed using geodesic shooting, which requires to characterize the Jacobi fields of $\mathcal M$. The particular case of curves lying in the hyperbolic half-plane $\mathbb H$ is considered as an example, in the setting of radar signal processing.
\end{abstract}

\maketitle

\section{Introduction}

Computing distances between shapes of open or closed curves is of interest in many applications, from medical imaging to radar signal processing, as soon as one wants to compare, classify or statistically analyze trajectories or contours of objects. While the shape of an organ or the trajectory of an object on a short distance can be modeled by a curve in the plane $\mathbb R^2$ or in the ambient space $\mathbb R^3$, some applications provide curves in an intrinsically non-flat space. Simple examples in positive curvature include trajectories on the sphere where points represent positions on the earth, and a negatively-curved space of interest in signal processing is the hyperbolic half plane, which as we will explain later coincides with the statistical manifold of Gaussian densities. We are motivated by the study of curves in the latter for signal processing purposes, however our framework is more general. 

Here we consider open oriented curves in a Riemannian manifold $M$, more precisely the space of smooth immersions $c : [0,1] \rightarrow M$,
\begin{equation*}
\mathcal{M}=\text{Imm}([0,1],M) = \{ c\in C^\infty([0,1],M),\, c'(t) \neq 0 \,\, \forall t\in [0,1] \}.
\end{equation*}
To compare or average elements of this space, one way to proceed is to equip $\mathcal M$ with a Riemannian structure, that is to locally define a scalar product $G$ on its tangent space $T\mathcal M$. A property that is usually required of this metric is reparametrization invariance, that is that the metric be the same at all points of $\mathcal M$ representing curves that are identical modulo reparametrization. Two curves are identical modulo reparametrization when they pass through the same points of $M$ but at different speeds. Reparametrizations are represented by increasing diffeomorphisms $\phi : [0,1] \rightarrow [0,1]$ (so that they preserve the end points of the curves), and their set is denoted by $\text{Diff}^+([0,1])$. Elements $h,k \in T_c\mathcal M$ of the tangent space in $c\in \mathcal M$ are infinitesimal deformations of $c$ and can be seen as vector fields along the curve $c$ in $M$ (this results from the so called "Exponential law" for smooth functions, see e.g. \cite{kri}, Theorem 5.6.). The Riemannian metric $G$ is reparametrization invariant if the action of $\text{Diff}^+([0,1])$ is isometric for $G$
\begin{equation}
\label{reparinv}
G_{c\circ \phi}(h\circ \phi, k\circ \phi)=G_c(h,k),
\end{equation}
for any $c\in \mathcal{M}$, $h,k \in T_c\mathcal{M}$ and $\phi \in \text{Diff}^+([0,1])$. This is often called the \emph{equivariance property}, and it guarantees that the induced distance between two curves $c_0$ and $c_1$ does not change if we reparametrize them by the same diffeomorphism $\phi$
\begin{equation*}
d(c_o\circ\phi,c_1\circ\phi) = d(c_0,c_1).
\end{equation*}
What's more, a reparametrization invariant metric on the space of curves induces a Riemannian structure on the "shape space", where the space of reparametrizations is quotiented out. A shape can be seen as the equivalence class of all the curves that are identical modulo a change of parameterization, and the shape space as the associated quotient space
\begin{equation*}
\mathcal{S}=\text{Imm}([0,1],M)/\text{Diff}^+([0,1]).
\end{equation*}
While the space of immersions is an open submanifold of the Fr\'echet manifold $C^\infty([0,1],M)$ (see \cite{mich0}, Theorem 10.4.), the shape space is not a manifold and therefore the fiber bundle structure we discuss next is to be understood formally. We get a principal bundle structure $\pi : \mathcal{M} \rightarrow \mathcal{S}$, which induces a decomposition of the tangent bundle $T\mathcal{M}=V\mathcal{M} \oplus H\mathcal{M}$ into a vertical subspace $V\mathcal{M}=\ker(T\pi)$ consisting of all vectors tangent to the fibers of $\mathcal{M}$ over $\mathcal{S}$, and a horizontal subspace $H\mathcal{M}=\left(V\mathcal{M}\right)^{\perp_G}$ defined as the orthogonal complement of $V\mathcal{M}$ according to the metric $G$ that we put on $\mathcal{M}$. If $G$ verifies the equivariance property, then it induces a Riemannian metric $\hat G$ on the shape space, for which the geodesics are the projected horizontal geodesics of $\mathcal M$ for $G$. The geodesic distance $\hat d$ on $\mathcal S$ between the shapes $[c_0]$ and $[c_1]$ of two given curves $c_0$ and $c_1$ is then given by 
\begin{equation*}
\hat d\left( [c_0] , [c_1] \right) = \inf \left\{\, d\left(c_0, c_1\circ \phi \right) \, | \, \, \phi \in \text{Diff}^+([0,1]) \, \right\},
\end{equation*}
and $\hat d$ verifies the stronger property
\begin{equation*}
\hat d(c_0\circ\phi,c_1\circ\psi)=\hat d(c_0,c_1),
\end{equation*}
for any reparametrizations $\phi,\psi\in\text{Diff}^+([0,1])$. This motivates the choice of a reparametrization invariant metric on $\mathcal M$.

Riemannian metrics on the space of curves lying in the Euclidean space $\mathbb R^n$, and especially closed curves $c:S^1 \rightarrow \mathbb R^n$ ($S^1$ is the circle), have been widely studied (\cite{younes1}, \cite{mich2}, \cite{mich3}, \cite{bauer2}). The most natural candidate for a reparametrization invariant metric is the $L^2$-metric with integration over arc length $\mathrm d\ell = \left\|c'(t)\right\| \mathrm dt$
\begin{equation*}
G^{L^2}_c(h,k) = \int \langle h , k \rangle \, \mathrm d\ell,
\end{equation*}
but Michor and Mumford have shown in \cite{mich1} that the induced metric on the shape space always vanishes. This has motivated the study of Sobolev-type metrics (\cite{mich3}, \cite{men}, \cite{younes2}, \cite{bauer1}), where higher order derivatives are introduced. Local existence and uniqueness of geodesics for first and second-order Sobolev metrics on the space of closed plane curves were  shown in \cite{mich3} and completion results were given in \cite{men}. One first-order Sobolev metric where different weights are given to the tangential and normal parts of the derivative has proved particularly interesting for the applications (\cite{laga}, \cite{su})
\begin{equation}
\label{sobolev}
G_c(h,k)=\int \langle \, D_{\ell}h^N,D_\ell k^N \,\rangle + \frac{1}{4} \langle \, D_\ell h^T, D_\ell k^T \, \rangle \,\, \mathrm d\ell.
\end{equation}
In that case $c$ is a curve in $\mathbb R^n$, $\langle \cdot,\cdot \rangle$ denotes the Euclidean metric on $\mathbb{R}^n$, $\, D_\ell h=h'/\left\| c' \right\| \,$ is the derivation of $h$ according to arc length, $D_\ell h^T=\langle D_\ell h,v\rangle v \,$ is the projection of $D_\ell h$ on the unit speed vector field $\, v=c'/\|c'\|$, and $\, D_sh^N=D_sh-D_sh^T $. This metric belongs to the class of so-called \emph{elastic} metrics, defined by
\begin{equation*}
G^{a,b}_c(h,k)=\int a^2\langle \, D_{\ell}h^N,D_\ell k^N \,\rangle + b^2 \langle \, D_\ell h^T, D_\ell k^T \, \rangle \,\, \mathrm d\ell,
\end{equation*}
for any weights $a, b \in \mathbb R_+$. The parameters $a$ and $b$ respectively control the degree of bending and stretching of the curve. Srivastava et al. introduced in \cite{sri} a convenient framework to study the case where $a=1$ and $b=1/2\,$ by showing that metric \eqref{sobolev} could be obtained by pullback of the $L^2$-metric via a simple transformation $R$ called the Square Root Velocity Function (SRVF), which associates to each curve its velocity renormalized by the square root of its norm. A similar idea had been previously introduced in \cite{younes1} and then used in \cite{younes2}, where a Sobolev metric is also mapped to an $L^2$-metric. The general elastic metric $G^{a,b}$ with weights $a$ and $b$ satisfying $4b^2\geq a^2$ can also be studied using a generalization of the SRVF \cite{bauer2}.

The SRV framework can be extended to curves in a Lie groupe using translations \cite{cell}, and to curves in a general manifold using parallel transport. For manifold-valued curves, this can be done in a way that enables to move the computations to the tangent space to the origin of one of the two curves under comparison \cite{lb}, \cite{su}, \cite{zhang}. In \cite{lb} the authors consider the general elastic metric $G^{a,b}$, but no Riemannian framework is given. In \cite{zhang}, a Riemannian framework is given for the case $a=1$, $b=1/2$, and the geodesic equations are derived. In this paper we also restrict to this particular choice of coefficients $a$ and $b$ for simplicity, but we propose another generalization of the SRV framework to manifold-valued curves. Instead of encoding the information of each curve within a tangent space at a single point as in \cite{lb} and \cite{zhang} using parallel transport, the distance is computed in the manifold itself which enables us to be more directly dependent on its geometry. Intuitively, the data of each curve is no longer concentrated at any one point, and so the energy of the deformation between two curves takes into account the curvature of the manifold along the entire "deformation surface", not just along the path traversed by the starting point of the curve.

In the following section, we introduce our metric as the pullback of a quite natural metric on the tangent bundle $T\mathcal M$, and show that it induces a fiber bundle structure over the manifold $M$ seen as the set of starting points of the curves. In section $3$, we give the induced geodesic distance and highlight the difference with respect to the distance introduced in \cite{zhang}. In section 4, we give the geodesic equations associated to our metric and exploit them to build the exponential map. Geodesics of the space of curves can then be computed using geodesic shooting. To this end, we describe the Jacobi fields on $\mathcal M$. We test these algorithms on curves lying in the hyperbolic half-plane $\mathbb H$, a choice that we motivate in section 5. Finally, in the setting of radar spectral analysis, we model locally stationary radar signals by curves in $\mathbb H$ and compute their mean.

\section{Extension of the SRV framework to manifold-valued curves}

\subsection{Our metric on the space of curves} 

Let $c : [0,1] \rightarrow M$ be a curve in $M$ and $h,k \in T_c\mathcal M$ two infinitesimal deformations. We consider the following first-order Sobolev metric on $\mathcal M$
\begin{equation*}
G_c(h,k) = \Braket{ \, h(0) , k(0) \, } + \int \Braket {\nabla_\ell h^N,\nabla_\ell k^N }  + \frac{1}{4} \Braket{ \nabla_\ell h^T, \nabla_\ell k^T } \, \mathrm d\ell,
\end{equation*}
where we integrate according to arc length $\mathrm d\ell = \|c'(t)\|\mathrm dt$, $\langle \cdot, \cdot \rangle$ and $\nabla$ respectively denote the Riemannian metric and the associated Levi-Civita connection of the manifold $M$, $\nabla_\ell h=\frac{1}{\left\| c' \right\|}\nabla_{c'}h$ is the covariant derivative of $h$ according to arc length, and $\nabla_\ell h^T=\langle \nabla_\ell h,v\rangle v \,$ and $\, \nabla_\ell h^N=\nabla_\ell h-\nabla_\ell h^T$ are its tangential and normal components respectively, with the notation $v=c'/\|c'\|$. If $M$ is a flat Euclidean space, we obtain the metric \eqref{sobolev} studied in \cite{sri}, with an added term involving the origins. Without this extra term, the bilinear form $G$ is not definite since it vanishes if $h$ or $k$ is covariantly constant along $c$. Here we show that $G$ can be obtained as the pullback of a very natural metric $\tilde G$ on the tangent bundle $T\mathcal M$. We consider the square root velocity function (SRVF, introduced in \cite{sri}) on the space of curves in $M$,
\begin{equation*}
R : \mathcal{M} \rightarrow T\mathcal{M}, \quad c \mapsto \frac{c'}{\sqrt{\left\|c'\right\|}},
\end{equation*}
where $\left\| \cdot \right\|$ is the norm associated to the Riemannian metric on $M$. In order to define $\tilde G$, we introduce the following projections from $TTM$ to $T M$. Let $\xi \in T_{(p,u)}T M$ and $t \mapsto (x(t),U(t))$ be a curve in $T M$ that passes through $(p,u)$ at time $0$ at speed $\xi$. Then we define the vertical and horizontal projections
\begin{eqnarray*}
\text{vp}_{(p,u)} &:& T_{(p,u)}T M \rightarrow T_pM, \quad \xi \mapsto \xi_V := \nabla_{x'(0)}U, \\
\text{hp}_{(p,u)} &:& T_{(p,u)}T M \rightarrow T_pM, \quad \xi \mapsto \xi_H := x'(0).
\end{eqnarray*}
The horizontal and vertical projections live in the tangent bundle $TM$ and are not to be confused with the horizontal and vertical parts which live in the double tangent bundle $TTM$ and will be denoted by $\xi^H$, $\xi^V$. Furthermore, let us point out that the horizontal projection is simply the differential of the natural projection $T M \rightarrow M$, and that according to these definitions, a very natural metric on the tangent bundle $TM$, the Sasaki metric (\cite{sas1}, \cite{sas2}), can be written
\begin{equation*}
g^{S}_{(p,u)}(\xi,\eta) = \Braket{\, \xi_H\, ,\, \eta_H \,} + \Braket{\, \xi_V\, ,\, \eta_V \,},
\end{equation*}
where $\langle \cdot,\cdot \rangle$ denotes the Riemannian metric on $M$. Now we can define the metric that we put on $T\mathcal{M}$. Let us consider $\, h \in T\mathcal{M}\,$ and $\, \xi, \eta \in T_hT\mathcal{M \,}$. We define
\begin{equation*}
\tilde G_h\left(\xi,\eta\right) \,= \, \Braket{\, \xi(0)_H\, ,\, \eta(0)_H } \,+\, \int_0^1 \Braket{ \, \xi(t)_V \, , \, \eta(t)_V } \, \mathrm dt,
\end{equation*}
where $\, \xi(t)_H=\text{hp}(\xi(t))\in TM \,$ and $\, \xi(t)_V=\text{vp}(\xi(t))\in TM \,$ are the horizontal and vertical projections of  $\, \xi(t) \in TTM \,$ for all $t$. Then we have the following result.
\begin{proposition}
The metric $G$ on the space of curves $\mathcal M$ can be obtained as pullback of the metric $\tilde G$ by the square root velocity function $R$, that is
\begin{equation*}
G_c(h,k) = \tilde G_{R( c)}\left( T_cR(h) , T_cR(k) \right),
\end{equation*}
for any curve $c\in \mathcal{M}$, and vector fields $h,k \in T_c\mathcal{M}$ along $c$.
\end{proposition}
\begin{notation}
Here and in all the paper we will denote by $s$ the parameter of paths in the space of curves $\mathcal M$ and by $t$ the parameter of a curve in $M$. For any path of curves $\, s \mapsto c(s,\cdot) \,$ the corresponding derivatives will be denoted by $\, c_s=\partial c/\partial s \,$ and $\, c_t=\partial c/\partial t \,$, and we will also use the notations $\nabla_s = \nabla_{\partial c/\partial s}$ and $\nabla_t = \nabla_{\partial c/\partial t}$.
\end{notation}
\begin{proof}[Proof of Proposition 1]
For any $t\in[0,1]$, we have $T_cR(h)(t)_H=h(t)$ and $T_cR(h)_V=\nabla_hR(c )(t)$. To prove this proposition, we just need to compute the latter. Let $\, s \mapsto c(s,\cdot) \,$ be a curve in $\mathcal{M}$ such that $\, c(0,\cdot)=c \,$ and $\, c_s(0,\cdot)=h \,$. 
Then
\begin{eqnarray*}
\nabla_hR(c )(t)&=& \frac{1}{\left\|c'\right\|^{1/2}} \nabla_h c' + h\left(\left\|c' \right\|^{-1/2}\right)c' \\
&=& \frac{1}{\left\| c_t\right\|^{1/2}} \nabla_{s} c_t+ \partial_s \Braket{\, c_t \, , \, c_t \,}^{-1/4} c_t \\
&=& \frac{1}{\left\| c_t\right\|^{1/2}} \nabla_t c_s - \frac{1}{2} \Braket{ \, c_t\, ,\,c_t\,}^{-5/4} \Braket{ \, \nabla_s c_t \, , \,c_t \,} \, c_t \\
&=& \left\|c'\right\|^{1/2} \left( \left( \nabla_\ell h\right)^N + \frac{1}{2} \langle \nabla_\ell h\, ,\, \frac{c'}{\|c'\|}\rangle \frac{c'}{\|c'\|} \right),
\end{eqnarray*}
where we used twice the inversion $\nabla_s c_t=\nabla_tc_s$.
\end{proof}

\subsection{Fiber bundle structures}

This choice of metric induces two fiber bundle structures. While the second one is an actual fiber bundle structure between two manifolds, the first structure is over the shape space which, as discussed in the introduction, is not a manifold, and so it should be understood formally.  Note that we could obtain a manifold structure by restricting ourselves to so-called "free" immersions, that is elements of $\mathcal M$ on which the diffeomorphism group acts freely, see \cite{mich2}.
\subsubsection*{Principal bundle over the shape space} Just as in the planar case, the fact that the square root velocity function $R$ satisfies
\begin{equation*}
R(c\circ \phi) = \sqrt{\phi'} \left(R(c ) \circ \phi \right), 
\end{equation*}
for all $c \in \mathcal{M}$, $h,k \in T_c\mathcal{M}$ and $\phi \in \text{Diff}^+([0,1])$, guarantees that the integral part of $G$ is reparametrization invariant. Remembering that the reparametrizations $\phi\in \text{Diff}^+([0,1])$ preserve the origins of the curves, we notice that $G$ is constant along the fibers and verifies the equivariance property \eqref{reparinv}. We then have a formal principal bundle structure over the shape space 
\begin{equation*}
\pi : \mathcal{M}=\text{Imm}([0,1],M) \rightarrow \mathcal{S}=\mathcal{M}/ \text{Diff}^+([0,1]).
\end{equation*}
which induces a decomposition $\, T\mathcal{M}=V\mathcal{M} \overset{\perp}{\oplus} H\mathcal{M} \,$. There exists a Riemannian metric $\hat G$ on the shape space $\mathcal{S}$ such that $\pi$ is (formally) a Riemannian submersion from $(\mathcal{M},G)$ to $(\mathcal{S},\hat G)$
\begin{equation*}
G_c(h^H,k^H) = \hat G_{\pi( c)}\left( T_c\pi(h), T_c\pi(k) \right),
\end{equation*}
where $h^H$ and $k^H$ are the horizontal parts of $h$ and $k$, as well as the horizontal lifts of $T_c\pi(h)$ and $T_c\pi(k)$, respectively. This expression does in fact define $\hat G$ in the sense that it does not depend on the choice of the representatives $c$, $h$ and $k$. For more details, see the theory of Riemannian submersions and G-manifolds in \cite{mich4}.
\subsubsection*{Fiber bundle over the starting points}
The special role played by the starting point in the metric $G$ induces another fiber bundle structure, where the base space is the manifold $M$, seen as the set of starting points of the curves, and the fibers are composed of the curves which have the same origin. The projection is then
\begin{equation*}
\pi^{(*)} : \mathcal{M} \rightarrow M, \quad c \mapsto c(0).
\end{equation*}
It induces another decomposition of the tangent bundle in vertical and horizontal bundles
\begin{eqnarray*}
V^{(*)}_c\mathcal{M} &=& \ker T\pi^{(*)} = \left\{ \, h \in T_c\mathcal{M} \, | \, h(0)=0 \, \right\}, \\
H^{(*)}_c\mathcal{M} &=& \left( V^{(*)}_c\mathcal{M}\right)^{\perp_G}.
\end{eqnarray*}
\begin{proposition}
We have the usual decomposition $T\mathcal{M}=V^{(*)}\mathcal{M} \,\, \overset{\perp}{\oplus} \,\, H^{(*)}\mathcal{M}$, the horizontal bundle $H^{(*)}_c\mathcal{M}$ consists of parallel vector fields along $c$, and $\pi^{(*)}$ is a Riemannian submersion for $(\mathcal{M},G)$ and $(M, \langle \cdot, \cdot \rangle)$.  
\end{proposition}
\begin{proof}
Let $h$ be a tangent vector. Consider $h_0$ the parallel vector field along $c$ with initial value $h_0(0)=h(0)$. It is a horizontal vector, since its vanishing covariant derivative along $c$ assures that for any vertical vector $l$ we have $G_c(h_0,l)=0$. The difference $\tilde h=h-h_0$ between those two vectors has initial value $0$ and so it is a vertical vector, which gives a decomposition of $h$ into a horizontal vector and a vertical vector. The definition of $H^{(*)}\mathcal M$ as the orthogonal complement of $V^{(*)}\mathcal M$ guaranties that their sum is direct. Now if $k$ is another tangent vector, then the scalar product between their horizontal parts is
\begin{equation*}
\begin{split}
G_c(h^H,k^H)=\big\langle \, h^H(0),k^H(0) \, &\big\rangle_{c(0)} \\
&=\big\langle \, h(0), k(0) \, \big\rangle_{c(0)} = \big\langle \, T_c\pi^{(*)}(h^H) ,T_c\pi^{(*)}(k^H) \, \big\rangle_{\pi^{(*)}},
\end{split}
\end{equation*}
which proves that $\pi^{(*)}$ is a Riemannian submersion and completes the proof.
\end{proof}

\section{Induced distance on the space of curves}

Here we give an expression of the geodesic distance induced by the metric $G$. We show that it can be written similarly to the product distance given in \cite{lb} and \cite{zhang}, with an added curvature term. Let us consider two curves $c_0, c_1 \in \mathcal{M}$, and a path of curves $s \mapsto c(s,\cdot)$ linking them in $\mathcal{M}$ \begin{equation*}
c(0,t)=c_0(t), \quad c(1,t)=c_1(t),
\end{equation*} 
for all $t \in [0,1]$. We denote by $q(s,\cdot)=R\left(c(s,\cdot)\right)$ the image of this path of curves by the SRVF $R$. Note that $q$ is a vector field along the surface $c$ in $M$. Let now $\tilde q$ be the "raising" of $q$ in the tangent space $T_{c(0,0)}M$ defined by
\begin{equation}
\label{raisingq}
\tilde q(s,t) = P_{c(\cdot,0)}^{s,0}\circ P_{c(s,\cdot)}^{t,0} \left( q(s,t) \right),
\end{equation}
where we denote by $P_\gamma^{t_1,t_2}:T_{\gamma(t_1)}M \rightarrow T_{\gamma(t_2)}M$ the parallel transport along a curve $\gamma$ from $\gamma(t_1)$ to $\gamma(t_2)$. Notice that $\tilde q$ is a surface in a vector space, as illustrated in Figure \ref{fig:dist}. Lastly, we introduce a vector field $(a,\tau)\mapsto \omega^{s,t}(a,\tau)$ in $M$, which parallel translates $q(s,t)$ along $c(s,\cdot)$ to its origin, then along $c(\cdot,0)$ and back down again, as shown in Figure \ref{fig:dist}. More precisely
\begin{equation}
\label{omega}
\omega^{s,t}(a,\tau)=P_{c(a,\cdot)}^{0,\tau} \circ P_{c(\cdot,0)}^{s,a} \circ P_{c(s,\cdot)}^{t,0} \left( q(s,t) \right),
\end{equation}
for all $b,s$. That way the quantity $\nabla_a\omega^{s,t}(s,t)$ measures the holonomy along the rectangle of infinitesimal width shown in Figure \ref{fig:dist}. 
\begin{proposition}
\label{distance1}
With the above notations, the geodesic distance induced by the Riemannian metric $G$ between two curves $c_0$ and $c_1$ on the space $\mathcal{M}=\textup{Imm}([0,1],M)$ of parameterized curves is given by
\begin{equation}
\label{eqdist1a}
d(c_0,c_1)= \inf_{c(0,\cdot)=c_0, c(1,\cdot)=c_1} \int_0^1 \sqrt{\left\| c_s(s,0) \right\|^2 + \int_0^1 \left\| \nabla_sq(s,t) \right\|^2 \, \mathrm dt} \,\,\, \mathrm ds,
\end{equation}
where $q=R(c)$ is the Square Root Velocity representation of the curve $c$ and the norm is the one associated to the Riemannian metric on $M$. It can also be written as a function of the "raising" $\tilde q$ of $q$ in the tangent space $T_{c_0(0)}M$ defined by \eqref{raisingq}, 
\begin{equation}
\label{eqdist1b}
d(c_0,c_1)= \inf_{c(0,\cdot)=c_0, c(1,\cdot)=c_1} \int_0^1 \sqrt{\left\| c_s(s,0) \right\|^2 + \int_0^1 \left\| \tilde q_s(s,t) +  \Omega(s,t) \right\|^2 \mathrm dt} \,\,\, \mathrm ds,
\end{equation}
where $\Omega$ is a curvature term measuring the holonomy along a rectangle of infinitesimal width
\begin{eqnarray*}
\Omega(s,t)&=&P_{c(\cdot,0)}^{s,0} \circ P_{c(s, \cdot)}^{t,0} \left(\nabla_a\omega^{s,t}(s,t) \right)\\
&=& P_{c(\cdot,0)}^{s,0} \left( \int_0^t P_{c(s,\cdot)}^{\tau,0}\left( \mathcal R( c_\tau, c_s) P_{c(s,\cdot)}^{t,\tau}q(s,t) \right) \, \mathrm d\tau\right),
\end{eqnarray*}
if $\mathcal R$ denotes the curvature tensor of the manifold $M$ and $\omega^{s,t}$ is defined by \eqref{omega}.
\end{proposition}
\begin{figure}[h]
\centering
\includegraphics[width=8.2cm]{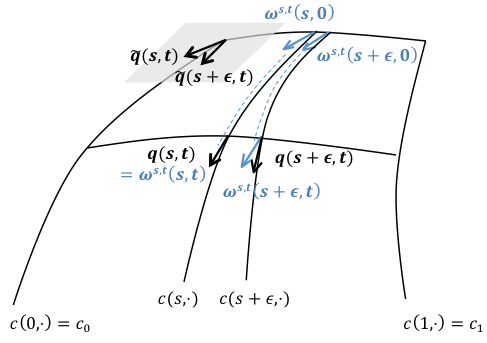}
\caption{Illustration of the distance between two curves $c_0$ and $c_1$ in the space of curves $\mathcal{M}$.}
\label{fig:dist}
\end{figure} 
\begin{remark}
The second expression \eqref{eqdist1b} highlights the difference with respect to the distance given in \cite{lb} and \cite{zhang}. In the first term under the square root we can see the velocity vector of the curve $c(\cdot,0)$ linking the two origins, and in the second the velocity vector of the curve $\tilde q$ linking the TSRVF-images of the curves -- Transported Square Root Velocity Function, as introduced by Su et al. in \cite{su}. If instead we equip the tangent bundle $\text{T}\mathcal{M}$ with the metric
\begin{equation*}
\tilde G'_h(\xi,\xi) = \left\| \xi(0)_H \right\|^2 + \int_0^1 \left\| \, \xi(t)_V - \int_0^t P_c^{\tau,t}\left( \mathcal R(c', \xi_H) P_c^{t,\tau}q(t) \right)\, \mathrm d\tau \, \right\|^2 \mathrm dt,
\end{equation*}
for $h \in T\mathcal{M}$ and $\, \xi, \eta \in T_hT\mathcal{M}$, then the curvature term $\Omega$ vanishes and the geodesic distance on $\mathcal{M}$ becomes
\begin{equation}
\label{eqdist2}
d'(c_0,c_1)=\inf_{c(0,\cdot)=c_0,c(1,\cdot)=c_1} \,\,\,\, \int_0^1 \sqrt{\left\| c_s(s,0) \right\|^2 + \left\| \tilde q_s(s,\cdot) \right\|_{L^2}^2 } \,\,\, \mathrm ds,
\end{equation}
which corresponds exactly to the geodesic distance introduced by Zhang et al. in \cite{zhang} on the space $\mathbb{C}=\cup_{p\in M} L^2([0,1],T_pM)$. The difference between the two distances \eqref{eqdist1a} and \eqref{eqdist2} resides in the curvature term $\Omega$, which measures the holonomy along the rectangle of infinitesimal width shown in Figure \ref{fig:dist}, and arises from the fact that in the first one, we compute the distance in the manifold, whereas in the second, it is computed in the tangent space to one of the origins of the curves. Therefore, the first one takes more directly into account the "relief" of the manifold between the two curves under comparison. For example, if there is a "bump" between two curves in an otherwise relatively flat space, the second distance \eqref{eqdist2} might not see it, whereas the first one \eqref{eqdist1a} should thanks to the curvature term.
\end{remark}
\begin{remark}
Let us briefly consider the flat case : if the manifold $M$ is flat, e.g. $M=\mathbb R^n$, then the two distances \eqref{eqdist1a} and \eqref{eqdist2} coincide. If two curves $c_0$ and $c_1$ in $\mathbb R^n$ have the same starting point $p$, the first summand under the square root vanishes and the distance becomes the $L^2$-distance between the two SRV representations $q_0=R(c_0)$ and $q_1=R(c_1)$. If two $\mathbb R^n$-valued curves differ only by a translation, then the distance is simply the distance between their origins. 
\end{remark}
\begin{remark}
Note that this distance is only local in general, that is, only works for curves that are "close enough". Indeed, if we consider two curves $c_1$, $c_2$ in $M=\mathbb R^2$ with the same origin, the distance between them is the length of the $L^2$ geodesic between their SRV representations $q_1$ and $q_2$ in $C^\infty([0,1], \mathbb R^2\backslash \{0\})$. If the minimizing geodesic between those two (in $C^\infty([0,1],\mathbb R^2)$) passes through $0$, then there is no minimizing geodesic between $q_1$ and $q_2$ in $C^\infty([0,1], \mathbb R^2\backslash \{0\})$.
\end{remark}
\begin{proof}[Proof of Proposition 3]
Since $G$ is defined by pullback of $\tilde G$ by the SRVF $R$, we know that the lengths of $c$ in $\mathcal{M}$ and of $q=R( c)$ in $\text{T}\mathcal{M}$ are equal and so that
\begin{equation*}
d(c_0,c_1)=\inf_{c(0,\cdot)=c_0,c(1,\cdot)=c_1} \,\,\, \int_0^1 \sqrt{\tilde G \left( q_s(s,\cdot),q_s(s,\cdot)\right)} \,\,\, \mathrm ds,
\end{equation*}
with 
\begin{equation*}
\tilde G \left( q_s(s,\cdot),q_s(s,\cdot)\right) = \left\| c_s(s,0) \right\|^2 + \int_0^1 \left\| \nabla_sq(s,t) \right\|^2 \, \mathrm dt.
\end{equation*}
To obtain the second expression of this distance we need to express $\nabla_sq$ as a function of the derivative $\tilde q_s$. Let us fix $t \in [0,1]$, and consider the vector field $\nu$ along the surface $(s,\tau)\mapsto c(s,\tau)$ that is parallel along all curves $c(s,\cdot)$ and takes value $\nu(s,t)=q(s,t) \,$ in $\tau=t \,$ for any $s\in[0,1]$, that is
\begin{equation*}
\nu(s,\tau)=P_{c(s,\cdot)}^{t,\tau}\left(q(s,t)\right),
\end{equation*}
for all $s\in [0,1]$ and $\tau\in [0,1]$. With this definition we have $\nabla_s\nu(s,t)=\nabla_sq(s,t)$. Since $\nu(\cdot,0) : s \mapsto P_{c(s,\cdot)}^{t,0}\left(q(s,t)\right)$ is a vector field along $c(\cdot, 0)$, we can write
\begin{equation*}
\nabla_s\nu(s,0)=\nabla_s\left( P_{c(s,\cdot)}^{t,0}q(s,t) \right)=P_{c(\cdot, 0)}^{0,s} \bigg( \frac{\partial}{\partial s} P_{c(\cdot, 0)}^{s,0} \circ P_{c(s,\cdot)}^{t,0} \left(q(s,t)\right) \bigg) = P_{c(\cdot, 0)}^{0,s}\tilde q_s (s,t).
\end{equation*}
Noticing that we additionally have $\nabla_\tau\nu(s,\tau)=0$ for all $s,\tau \in[0,1]$, and using $\nabla_\tau\nabla_s\nu = \nabla_s\nabla_\tau\nu + \mathcal R(c_\tau, c_s)\nu$, the covariant derivative in $\tau=t$ can be written
\begin{eqnarray}
\label{nablav}
\nabla_s\nu(s,t)&=&P_{c(s,\cdot)}^{0,t} \left( \nabla_s\nu(s,0) \right) + \int_0^t P_{c(s,\cdot)}^{\tau,t}\left( \nabla_\tau\nabla_s \nu(s,\tau) \right) d\tau \notag \\
&=& P_{c(s,\cdot)}^{0,t} \circ P_{c(\cdot, 0)}^{0,s} \left( \tilde q_s(s,t) \right) + \int_0^t P_{c(s,\cdot)}^{\tau,t}\left( \mathcal R(c_\tau,c_s) P_{c(s,\cdot)}^{t,\tau}q(s,t) \right) \, \mathrm d\tau.
\end{eqnarray}
Now let us fix $s\in [0,1]$ as well. Notice that the vector field $\omega^{s,t}$ defined above as $\omega^{s,t}(a,\tau)=P_{c(a,\cdot)}^{0,\tau} \circ P_{c(\cdot,0)}^{s,a} \circ P_{c(s,\cdot)}^{t,0} \left( q(s,t) \right)$ verifies
\begin{eqnarray}
\nabla_\tau\omega^{s,t}(s,\tau)&=& 0 \quad \forall \tau \in [0,1],\label{nablaomega1}\\ 
\nabla_a\omega^{s,t}(a,0)&=& 0 \quad \forall a \in [0,1].\label{nablaomega2}
\end{eqnarray}
Note that unlike $\nu$, we do \textit{not} have $\nabla_a\omega^{s,t}(s,t)=\nabla_sq(s,t)$ because $\omega^{s,t}(a,t)=q(a,t)$ is only true for $a=s$. Using Equations \eqref{nablaomega1} and \eqref{nablaomega2} we get
\begin{align*}
\nabla_a\omega^{s,t}(s,t) &= P_{c(s,\cdot)}^{0,t}\left(\nabla_a\omega^{s,t}(s,0)\right) + \int_0^t P_{c(s,\cdot)}^{\tau,t}\left(\nabla_\tau\nabla_a \omega^{s,t}(s,\tau)\right) \mathrm d\tau \\
&= \int_0^t P_{c(s,\cdot)}^{\tau,t}\left( \nabla_a\nabla_\tau\omega^{s,t}(s,\tau) + \mathcal R(c_\tau, c_s)\omega^{s,t}(s,\tau) \right) \, \mathrm d\tau, \\
&= \int_0^t P_{c(s,\cdot)}^{\tau,t}\left( \mathcal R(c_\tau, c_s) P_{c(s,\cdot)}^{t,\tau}q(s,t) \right) \, \mathrm d\tau,
\end{align*}
which is the same integral as the one in \eqref{nablav}. Finally, since $\| \nabla_sq(s,t)\| = \|\nabla_s\nu(s,t)\| = \| P_{c(\cdot,0)}^{s,0} \circ P_{c(s, \cdot)}^{t,0} \left(\nabla_s\nu(s,t)\right)\|$ we obtain
\begin{align*}
\| \nabla_sq(s,t)\| &= \| \tilde q_s(s,t) + P_{c(\cdot,0)}^{s,0} \circ P_{c(s, \cdot)}^{t,0} \left(\nabla_a\omega^{s,t}(s,t)\right)\| \\
&= \left\| \tilde q_s(s,t) + P_{c(\cdot,0)}^{s,0} \int_0^t P_{c(s,\cdot)}^{\tau,0}\left( \mathcal R(c_\tau, c_s) P_{c(s,\cdot)}^{t,\tau}q(s,t) \right) \, \mathrm d\tau \right\|
\end{align*}
which gives Equation \eqref{eqdist1b} and completes the proof.
\end{proof}

\section{Computing geodesics}

\subsection{Geodesic equations on $\mathcal M$}

To be able to compute the distance given by \eqref{eqdist1a} between two curves, we first need to compute the optimal deformation $s\mapsto c(s,\cdot)$ from one to the other. That is, we need to characterize the geodesics of $\mathcal M$ for our metric. In order to do so, taking inspiration from \cite{zhang}, we use the variational principle. In what follows, we use the lighter notation $u(t_1)^{t_1,t_2}=P_c^{t_1,t_2}(u(t_1))$ to denote the parallel transport of a vector $u(t_1)\in T_{c(t_1)}M$ along a curve $c$ from $c(t_1)$ to $c(t_2)$, when there is no ambiguity on the choice of $c$. We also denote by $w^T=\langle w, v\rangle v$ the tangential component of any vector field $w$ along a curve $c$, that is its projection on the unit speed vector field $v=c'/\|c'\|$.
\begin{proposition}
Let $[0,1]\ni s \mapsto c(s,\cdot)\in\mathcal M$ be a path of curves. It is a geodesic of $\mathcal M$ if and only if it verifies the following equations
\begin{subequations}
\begin{align}
\nabla_sc_s(s,0) + r(s,0) \,\, =& \,\,\, 0 , \quad \forall s \label{geodeq1} \\
\nabla_s\nabla_s q(s,t) + \left\|q(s,t) \right\| \left( r(s,t) + r(s,t)^T \right) \, =& \,\,\, 0 , \quad \forall t,s \label{geodeq2}
\end{align}
\end{subequations}
where $q = c_t/\sqrt{\|c_t\|}$ is the SRV representation of $c$, the vector field $r$ is given by
\begin{equation*}
r(s,t) = \int_t^1 \mathcal R(q,\nabla_sq)c_s(s,\tau)^{\tau,t} \mathrm d\tau,
\end{equation*} 
and $r^{T} = \Braket{r,v} v \,$ with $v=c_t / \| c_t \|$, is the tangential component of $r$.
\end{proposition}
\begin{proof}
The path $c$ is a geodesic if and only if it is a critical point of the energy functional $E : \mathcal C^\infty([0,1],\mathcal M) \rightarrow \mathbb R_+$,
\begin{equation*}
E(c) = \frac{1}{2} \int_0^1 G\left(\frac{\partial c}{\partial s},\frac{\partial c}{\partial s}\right) \mathrm ds.
\end{equation*}
Let $a\mapsto \hat c(a,\cdot,\cdot)$, $a\in(-\epsilon, \epsilon)$, be a proper variation of the path $s\mapsto c(s,\cdot)$, meaning that it coincides with $c$ in $a=0$, and it preserves its end points
\begin{eqnarray*}
\hat c(0,s,t)&=&c(s,t) \quad \forall s,t, \\
\hat c_a(a,0,t)&=&0 \quad\quad \quad \forall a,t,\\
\hat c_a(a,1,t)&=&0 \quad \quad \quad \forall a,t.
\end{eqnarray*}
Then $c$ is a geodesic of $\mathcal M$ if and only if $\left. \frac{d}{da}\right|_{a=0}E(\hat c(a,\cdot,\cdot))=0$ for any proper variation $\hat c$. If we denote by $E(a) = E(\hat c(a,\cdot,\cdot))$, for $a\in(-\epsilon,\epsilon)$, the energy of a proper variation $\hat c$, then we have
\begin{equation*}
E(a)= \frac{1}{2} \int \left( \Braket{\,  \hat c_s(a,s,0), \hat c_s(a,s,0) \,} \mathrm ds \, + \, \int \Braket{\, \nabla_s \hat q(a,s,t),\nabla_s \hat q(a,s,t) \,} \mathrm dt \right)\, \mathrm ds,
\end{equation*}
where $\hat q=\hat c_t/\sqrt{\|\hat c_t\|}$ is the SRV representation of $\hat c$. Its derivative is given by
\begin{equation*}
E'(a) = \int \Braket{ \, \nabla_a  \hat c_s(a,s,0), \hat c_s(a,s,0) \,} \, \mathrm ds 
+ \int\int \Braket{\, \nabla_a \nabla_s \hat q(a,s,t),\nabla_s \hat q(a,s,t) \,} \, \mathrm dt \, \mathrm ds. 
\end{equation*}
Considering that the variation preserves the end points, integration by parts gives
\begin{align*}
\int \Braket{\, \nabla_a \hat c_s, \hat c_s \,} \mathrm ds &= - \int \Braket{\nabla_s \hat c_s, \hat c_a} \mathrm ds \\
\int \Braket{\, \nabla_s\nabla_a \hat q, \nabla_s \hat q \,} \mathrm ds &= - \int \Braket{\nabla_s \nabla_s \hat q, \nabla_a \hat q} \mathrm ds,
\end{align*}
and so we obtain
\begin{align*}
E'(a)&= - \, \int \left. \Braket{ \, \nabla_s  \hat c_s,  \hat c_a \,} \right|_{t=0} \mathrm ds
       \, + \, \int \int \Braket{ \, \mathcal R (  c_a,   c_s)   q \, +\, \nabla_s\nabla_a  q , \nabla_s  q\,} \mathrm dt \, \mathrm ds \\
       &= - \, \int \left. \Braket{ \, \nabla_s \hat c_s , \hat c_a \,} \right|_{t=0} \mathrm ds
       \,-\, \int \int \Braket{\, \mathcal R( \hat q, \nabla_s \hat q) \hat c_s , \hat c_a \,} \, + \, \Braket{\, \nabla_s \nabla_s \hat q , \nabla_a \hat q\,} \mathrm dt \, \mathrm ds.
\end{align*}
This quantity has to vanish in $a=0$ for all proper variations $\hat c$
\begin{equation*}
\begin{split}
\int \big\langle \left. \nabla_s  c_s \right|_{t=0}, &\left. \hat c_a \right|_{a=0, t=0} \,\big\rangle \, \mathrm ds \\
&+ \int \int \Braket{\, \mathcal R(q,\nabla_sq)c_s,\left.  \hat c_a \right|_{a=0} \,} \, + \, \Braket{\, \nabla_s\nabla_s q, \left. \nabla_a \hat q \right|_{a=0} \,} \, \mathrm dt \, \mathrm ds \, = \, 0. 
\end{split}
\end{equation*} 
We cannot yield any conclusions at this point, because $\hat c_a(0,s,t)$ and $\nabla_a \hat q(0,s,t)$ cannot be chosen independently, since $\hat q$ is not any vector field along $\hat c$ but its image via the Square Root Velocity Function. Computing the covariant derivative of $\hat q = \hat c_t / \|\hat c_t\|^{1/2}$ according to $a$ gives $\nabla_a\hat q = \|\hat c_t\|^{-1/2}(\nabla_a\hat c_t - \tfrac{1}{2}{\nabla_a\hat c_t}^T)$, and projecting both sides on $v=c_t/\|c_t\|$ results in ${\nabla_a\hat q}^T=\tfrac{1}{2} \|\hat q\|^{-1} {\nabla_a\hat c_t}^T$. We deduce
\begin{equation*}
\nabla_a\hat c_t = \left\| \hat q \right\| \left( \nabla_a \hat q + \nabla_a \hat q^T \right),
\end{equation*}
and since $\nabla_t\hat c_a = \nabla_a\hat c_t$, we can express the variation $\hat c_a$ as follows
\begin{equation*}
\hat c_a(0,s,t) = \hat c_a(0,s,0)^{0,t} + \int_0^t \left\| \hat q(0,s,\tau) \right\| \left( \nabla_a \hat q(0,s,\tau) + \nabla_a\hat q^T(0,s,\tau) \right)^{\tau,t} \mathrm d\tau.
\end{equation*}
Inserting this expression in the derivative of the energy we obtain the following, where we omit to write that the variations $\hat c$ and $\hat q$ are always taken in $a=0$ for the sake of readability,
\begin{equation*}
\begin{split}
& \int_0^1 \big\langle \, \nabla_s c_s(s,0) , \hat c_a(s,0) \,\big\rangle \, \mathrm ds + \int_0^1 \int_0^1 \big\langle\, \mathcal R(q, \nabla_s q) c_s(s,t), \hat c_a(s,0)^{0,t} \,\big\rangle \, \mathrm dt \, \mathrm ds \\
& + \int_0^1 \int_0^1 \bigg\langle \, \mathcal R(q, \nabla_s q) c_s(s,t) , \int_0^t \left\| \hat q(s,\tau) \right\| \left( \nabla_a \hat q(s,\tau) + \nabla_a\hat q^T(s,\tau) \right)^{\tau,t} \mathrm d\tau \, \bigg\rangle \, \mathrm dt \, \mathrm ds \\
& + \int_0^1 \int_0^1 \big\langle\, \nabla_s\nabla_s q(s,t) , \nabla_a \hat q(s,t) \,\big\rangle\, \mathrm dt \, \mathrm ds \\
=& \int_0^1 \bigg\langle\, \nabla_s c_s(s,0) + \int_0^1 \mathcal R(q, \nabla_s q)c_s(s,\tau)^{\tau,0} \mathrm d\tau \,, \hat c_a(s,0) \,\bigg\rangle\, \mathrm ds \\
& + \int_0^1 \int_0^1 \int_t^1 \big\langle \, \mathcal R(q, \nabla_s q) c_s(s,\tau)^{\tau,t} , \left\| \hat q(s,t) \right\| ( \nabla_a \hat q(s,t) + \nabla_a\hat q^T(s,t) ) \, \big\rangle \, \mathrm d\tau \, \mathrm dt \, \mathrm ds \\
& + \int_0^1 \int_0^1 \big\langle\, \nabla_s\nabla_s q(s,t) , \nabla_a \hat q(s,t) \,\big\rangle\, \mathrm dt \, \mathrm ds \\
=& \int_0^1 \bigg\langle \nabla_s c_s(s,0) + r(s,0) \,, \hat c_a(s,0) \bigg\rangle\, \mathrm ds \\
& + \int_0^1 \int_0^1 \big\langle \, \nabla_s\nabla_sq(s,t) + \|q(s,t)\| ( r(s,t) + r(s,t)^T ) , \nabla_a \hat q(s,t)\,\, \big\rangle \, \mathrm dt \, \mathrm ds \\
=& 0,
\end{split}
\end{equation*}
with the previously given definition of $r$. Since the variations $\hat c_a(0,s,0)$ and $\nabla_a \hat q(0,s,t)$ can be chosen independently and take any value for all $s$ and all $t$, we obtain the desired equations.
\end{proof}

\subsection{Exponential map}

Now that we have the geodesic equations, we are able to describe an algorithm which allows us to compute the geodesic $s\mapsto c(s,\cdot)$ starting from a point $c\in \mathcal M$ at speed $u\in T_c\mathcal M$. This amounts to finding the optimal deformation of the curve $c$ in the direction of the vector field $u$ according to our metric. We initialize this path $s\mapsto c(s,\cdot)$ by setting $c(0,\cdot)=c$ and $c_s(0,\cdot)=u$, and we propagate it using iterations of fixed step $\epsilon >0$. The aim is, given $c(s,\cdot)$ and $c_s(s,\cdot)$, to deduce $c(s+\epsilon,\cdot)$ and $c_s(s+\epsilon,\cdot)$. The first is obtained by following the exponential map on the manifold $M$ 
\begin{equation*}
c(s+\epsilon,t)=\exp^M_{c(s,t)} \left(\epsilon c_s(s,t)\right),
\end{equation*}
for all $t\in[0,1]$ and the second requires the computation of the variation $\nabla_s c_s(s,\cdot)$
\begin{equation*}
c_s(s+\epsilon,t)=\left[ c_s(s,t) + \epsilon \nabla_s c_s(s,t) \right]^{s,s+\epsilon},
\end{equation*}
for all $t\in[0,1]$ where once again, we use the notation $w(s)^{s,s+\epsilon}=P_c^{s,s+\epsilon} \left(w(s)\right)$ for the parallel transport of a vector field $s\mapsto w(s)$ along a curve $s\mapsto c(s)$ in $M$. If we assume that at time $s$ we have $c(s,\cdot)$ and $c_s(s,\cdot)$ at our disposal, then we can estimate $c_t(s,\cdot)$ and $\nabla_tc_s(s,\cdot)$, and deduce $q(s,\cdot)=c_t(s,\cdot)/\sqrt{\|c_t\|}$ as well as 
\begin{equation}
\label{nablasq}
\nabla_sq(s,\cdot)=\frac{\nabla_s c_t}{\sqrt{|c_t|}}(s,\cdot) - \frac{1}{2} \frac{\Braket{\, \nabla_s c_t, c_t \,}}{|c_t|^{5/2}}c_t(s,\cdot),
\end{equation}
using the fact that $\nabla_s c_t=\nabla_tc_s$. The variation $\nabla_sc_s(s,\cdot)$ can then be computed in the following way
\begin{equation}
\label{nscs}
\nabla_sc_s(s,t)=\nabla_sc_s(s,0)^{0,t} + \int_0^t \left[ \nabla_s\nabla_sc_t(s,\tau) + \mathcal R(c_t,c_s)c_s(s,\tau) \right]^{\tau,t} \mathrm d\tau
\end{equation}
for all $t\in[0,1]$, where $\nabla_sc_s(s,0)$ is given by equation \eqref{geodeq1}, the second order variation $\nabla_s\nabla_sc_t(s,\cdot)$ is given by
\begin{equation}
\begin{split}
\label{nablasnablasct}
\nabla_s \nabla_s c_t = |c_t|^{1/2} \nabla_s\nabla_sq \,+\, &\frac{ \Braket{\nabla_tc_s,c_t}}{|c_t|^2} \nabla_tc_s  \\
&+ \left( \frac{\Braket{\nabla_s\nabla_sq,c_t}}{|c_t|^{3/2}} \,-\, \frac{3}{2} \frac{\Braket{\nabla_tc_s,c_t}^2}{|c_t|^4} \,+\, \frac{|\nabla_tc_s|^2}{|c_t|^2} \right) c_t,
\end{split}
\end{equation}
and $\nabla_s\nabla_sq$ can be computed via equation \eqref{geodeq2}.
\begin{algorithm}[Exponential Map] 
\label{alg:expmap}
\leavevmode\par \noindent
Input : $c_0$, $u$. \\
Initialization : Set $c(0,t)=c_0(t)$ and $c_s(0,t)=u(t)$ for all $t\in[0,1]$. \\
Heredity : For $j =0, \hdots, m-1$, set $s=j\epsilon$ with $\epsilon = 1/m$ and
\begin{enumerate} 
\item compute for all $t$
\begin{align*}
c_t(s,t) &= \lim_{\delta \rightarrow 0} \frac{1}{\delta} \log^M _{c(s,t)} c(s,t+\delta), \\
\nabla_tc_s(s,t) &= \lim_{\delta \rightarrow 0} \frac{1}{\delta} \left( c_s(s,t+\delta)^{t+\delta,t} - c_s(s,t) \right),
\end{align*}
where $\,\, \log^M \,\,$ denotes the inverse of the exponential map on $M$, and compute $q(s,t)=\frac{1}{\sqrt{\|c_t\|}}c_t(s,t)$ and $\nabla_sq(s,t)$ using equation \eqref{nablasq}.

\medskip
\item Compute $r(s,t) = \int_t^1 \mathcal R(q,\nabla_sq)c_s(s,\tau)^{\tau,t} \mathrm d\tau,
$ and
\begin{equation*}
\nabla_s\nabla_s q(s,t) = - \, \|q(s,t)\| \left( r(s,t) + r(s,t)^T \right),
\end{equation*}
and deduce $\nabla_s\nabla_sc_t(s,t)$ using equation \eqref{nablasnablasct} for all $t\in[0,1]$.

\medskip
\item Initialize $\nabla_sc_s(s,0) = - r(s,0)$ and compute $\nabla_sc_s(s,\cdot)$ using equation \eqref{nscs}. 

\medskip
\item Finally, for all $t\in[0,1]$, set
\begin{align*}
c(s+\epsilon,t) &= \exp^M_{c(s,t)} \left( \epsilon c_s(s,t) \right), \\
c_s(s+\epsilon,t) &= \left[ \,c_s(s,t) + \epsilon \nabla_sc_s(s,t) \,\right]^{s,s+\epsilon},
\end{align*}
where $\exp^M$ is the exponential map on the manifold $M$.
\end{enumerate}
Output : $c = \exp^{\mathcal M}_{c_0} u$.
\end{algorithm}
The last step needed to compute the optimal deformation between two curves $c_0$ and $c_1$ is to find the appropriate initial speed $u$, that is the one that will connect $c_0$ to $c_1$. Since we do not have an explicit expression for this appropriate initial speed, we compute it iteratively using geodesic shooting.
\subsection{Geodesic shooting and Jacobi fields}

The aim of geodesic shooting is to compute the geodesic linking two points $p_0$ and $p_1$ of a manifold $\mathcal N$, knowing the exponential map $\exp^{\mathcal N}$. More precisely, the goal is to iteratively find the initial speed $u^{p_0,p_1}$ such that
\begin{equation*}
\exp_{p_0}^{\mathcal N}(u^{p_0,p_1}) = p_1. 
\end{equation*}
An initial speed vector $u \in T_{p_0}\mathcal N$ is chosen, and is iteratively updated after evaluating the gap between the point $p=\exp^{\mathcal N}_{p_0}u$ obtained by taking the exponential map at point $p_0$ in $u$ -- that is, by "shooting" from $p_0$ in the direction $u$ -- and the target point $p_1$. Assuming that the current point $p$ is "not too far" from the target point $p_1$, and that there exists a geodesic linking $p_0$ to $p_1$, we can consider that the gap between $p$ and $p_1$ is the extremity of a Jacobi field $J : [0,1] \rightarrow \mathcal N$ in the sense that it measures the variation between the geodesics $s\mapsto \exp_{p_0}^{\mathcal N}(s u)$ and $s\mapsto \exp_{p_0}^{\mathcal N}(s u^{p_0,p_1})$. Since both geodesics start at $p_0$, this Jacobi field has value $J(0)=0$ in $0$. Then, the current speed vector can be corrected by 
\begin{equation*}
u \leftarrow u + \dot J(0),
\end{equation*}
as shown in Figure \ref{fig:geodshoot}. Let us briefly explain why. If $c(a,s)$, $a\in(-\epsilon, \epsilon), s\in[0,1]$, is a family of geodesics starting from the same point $p_0$ at different speeds $u(a) \in T_{p_0}\mathcal N$, i.e. $c(a,s) = \exp^{\mathcal N}_{p_0} (s u(a))$, and $J(s) = c_a(0,s)$, $s\in[0,1]$ measures the way that these geodesics spread out, then we have
\begin{equation*}
\dot J(0) = \left.\frac{\partial}{\partial s}\right|_{s=0} \left.\frac{\partial}{\partial a}\right|_{a=0}\exp^{\mathcal N}_{p_0} (su(a))
= \left.\frac{\partial}{\partial a}\right|_{a=0} \left.\frac{\partial}{\partial s}\right|_{s=0}\exp^{\mathcal N}_{p_0} (su(a))
=\dot u(0).
\end{equation*}
In the context of geodesic shooting between two curves $c_0$ and $c_1$ in $\mathcal M$, the speed vector $u$ can be initialized using the $L^2$ logarithm map, the inverse of the exponential map for the $L^2$-metric (these maps are simply obtained by post-composition of mappings with the finite-dimensional maps $\exp^M$ and $\log^M$). That is, we set 
\begin{equation*}
u=\log_{c_0}^{L^2}(c_1).
\end{equation*}
The $L^2$ logarithm map also allows us to approximate the gap between the current point and the target point. This amounts to minimizing the functional $F(u)=\text{dist}_{L^2}(\exp^{\mathcal M}_{c_0}(u), c_1)$. We summarize as follows.
\begin{figure}
\centering
\includegraphics[width=7cm]{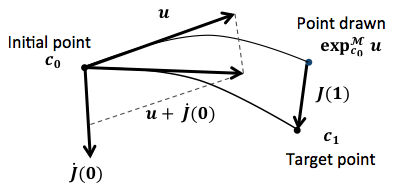}
\caption{Geodesic shooting in the space of curves $\mathcal M$}
\label{fig:geodshoot}
\end{figure}
\begin{algorithm}[Geodesic shooting]
\label{alg:geodshoot}
\leavevmode\par \noindent
Input : $c_0, c_1\in \mathcal M$. \\
Initialization : Set $u=\log^{L^2}_{c_0}(c_1)$.
Fix a threshold $\delta >0$.
\begin{enumerate}
\item Compute $c = \exp^{\mathcal M}_{c_0} (u)$ with Algorithm \ref{alg:expmap}.
\item Estimate the gap $j = \log^{L^2}_{c}(c_1)$.
\item If $\|j\|_{L^2} > \delta$, set $J(1) = j$ and $u \leftarrow u + \dot J(0)$ where $\dot J(0) = \phi^{-1}\left( J(1) \right)$ is computed using Algorithm \ref{alg:jacobi}, and go back to the first step. \\
Else, stop.
\end{enumerate}
Output : $c$ approximation of the geodesic linking $c_0$ and $c_1$.
\end{algorithm}
The function $\phi$ associates the last value $J(1)$ of a Jacobi field with initial value $J(0)=0$ to the initial speed $\dot J(0)$, and can be deduced from Algorithm \ref{alg:jacobi}, which describes the function associating $J(1)$ to the initial conditions $J(0)$ and $\dot J(0)$. To find the inverse of this function, we consider the image of a basis of the tangent vector space in which $\dot J(0)$ lives.
\begin{figure}
\subfloat{\includegraphics[width=0.32\textwidth]{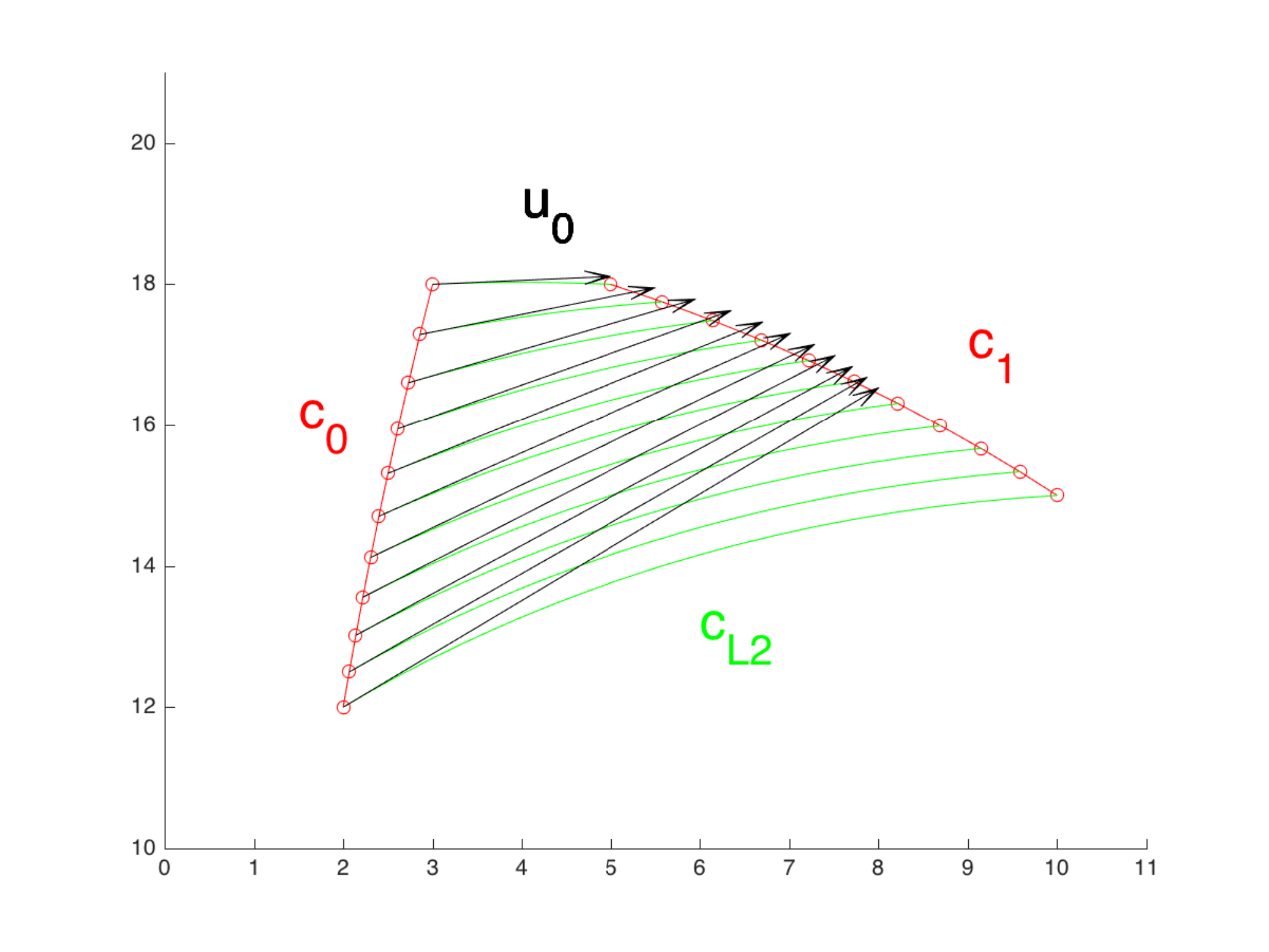}}
\subfloat{\includegraphics[width=0.32\textwidth]{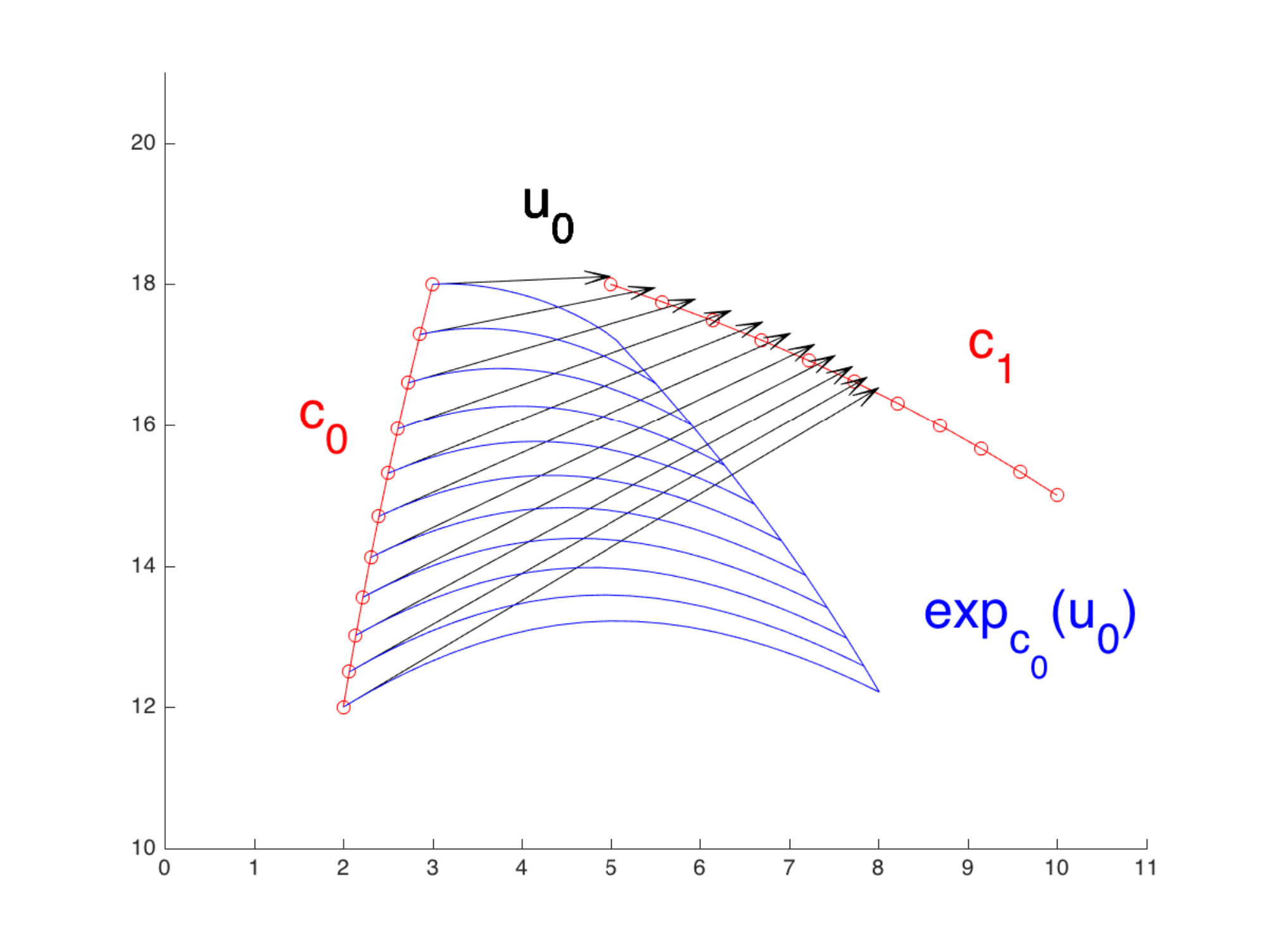}} 
\subfloat{\includegraphics[width=0.32\textwidth]{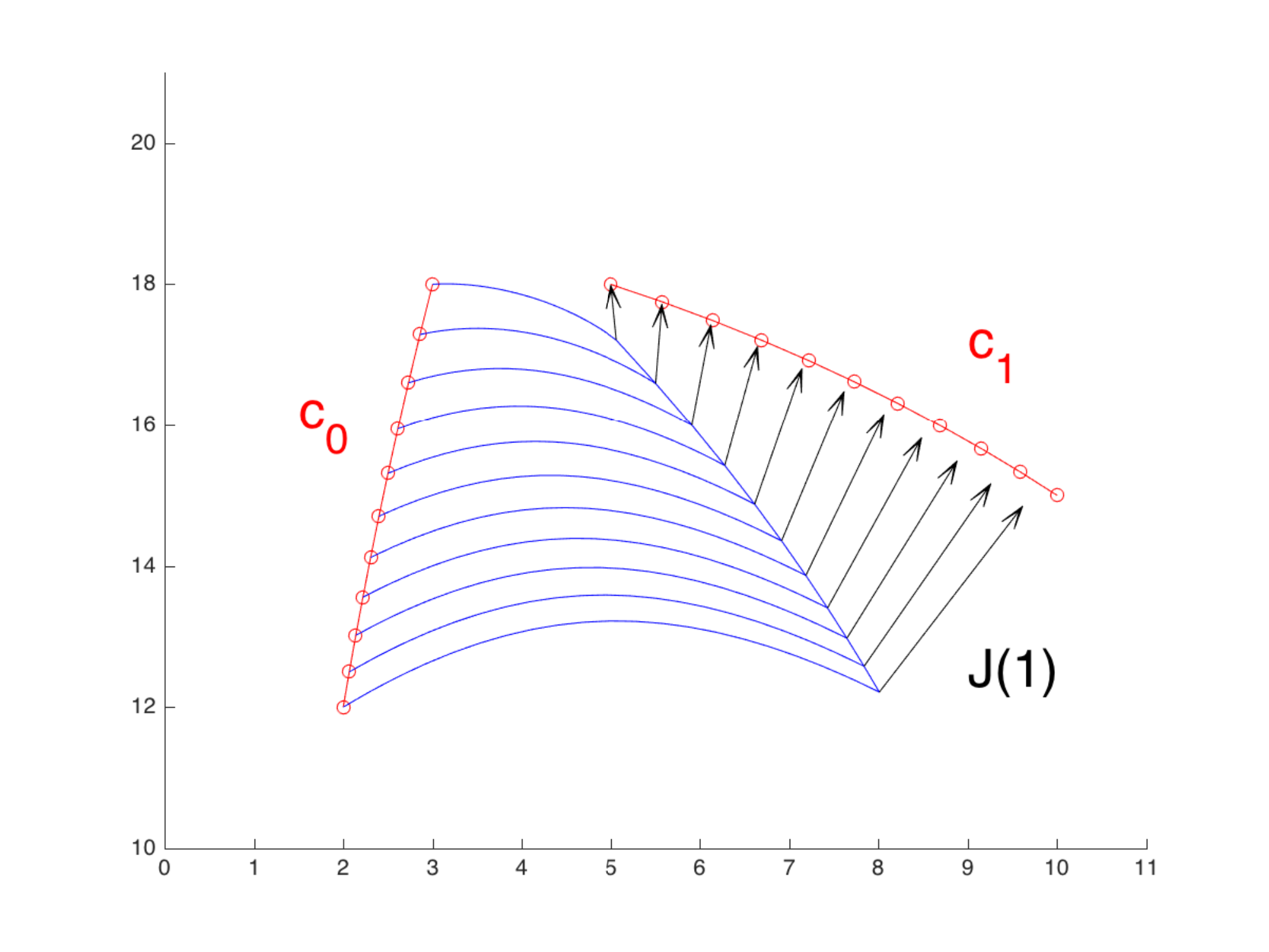}}\vspace{-0.7em}\\
\subfloat{\includegraphics[width=0.32\textwidth]{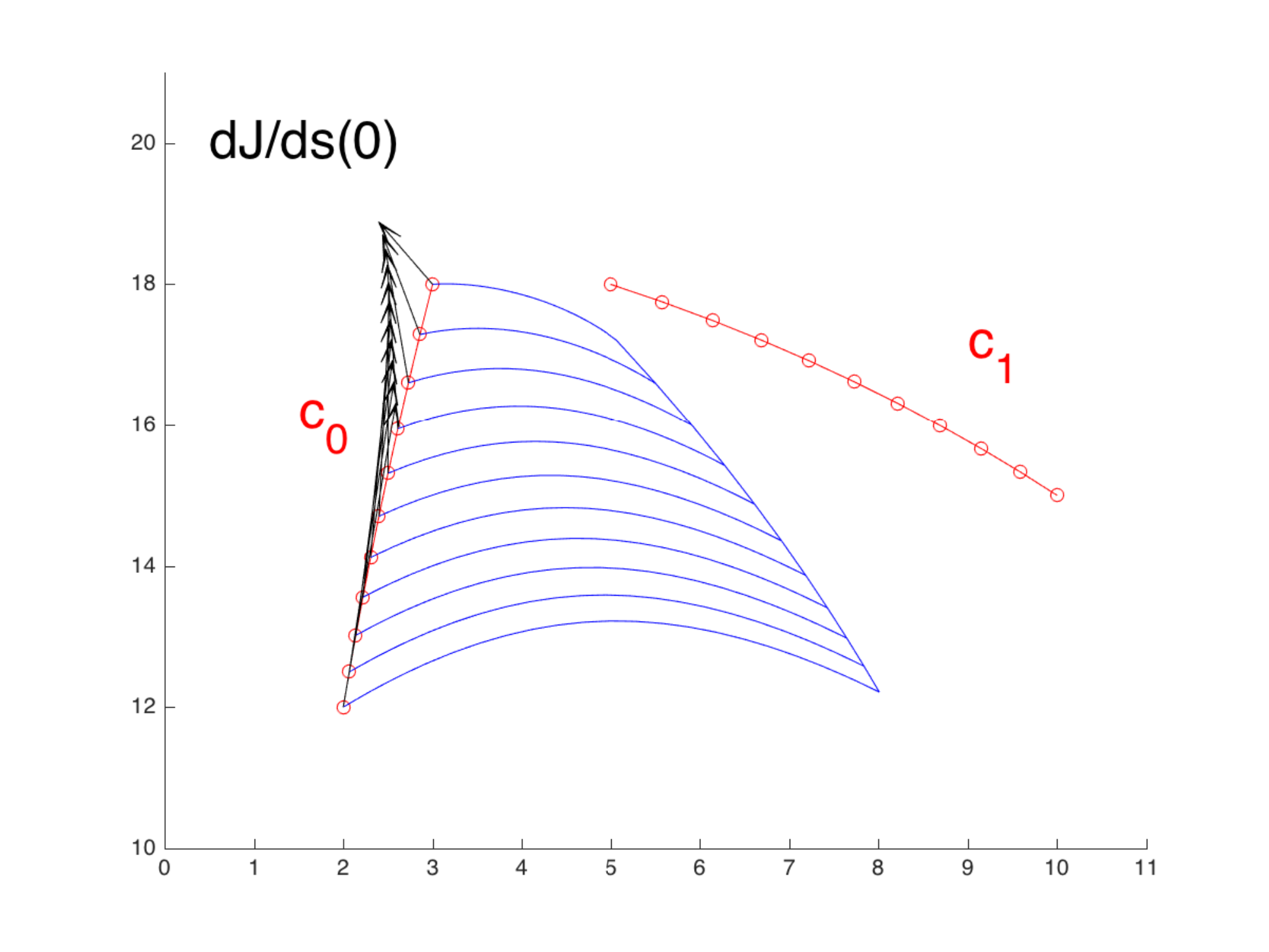}}
\subfloat{\includegraphics[width=0.32\textwidth]{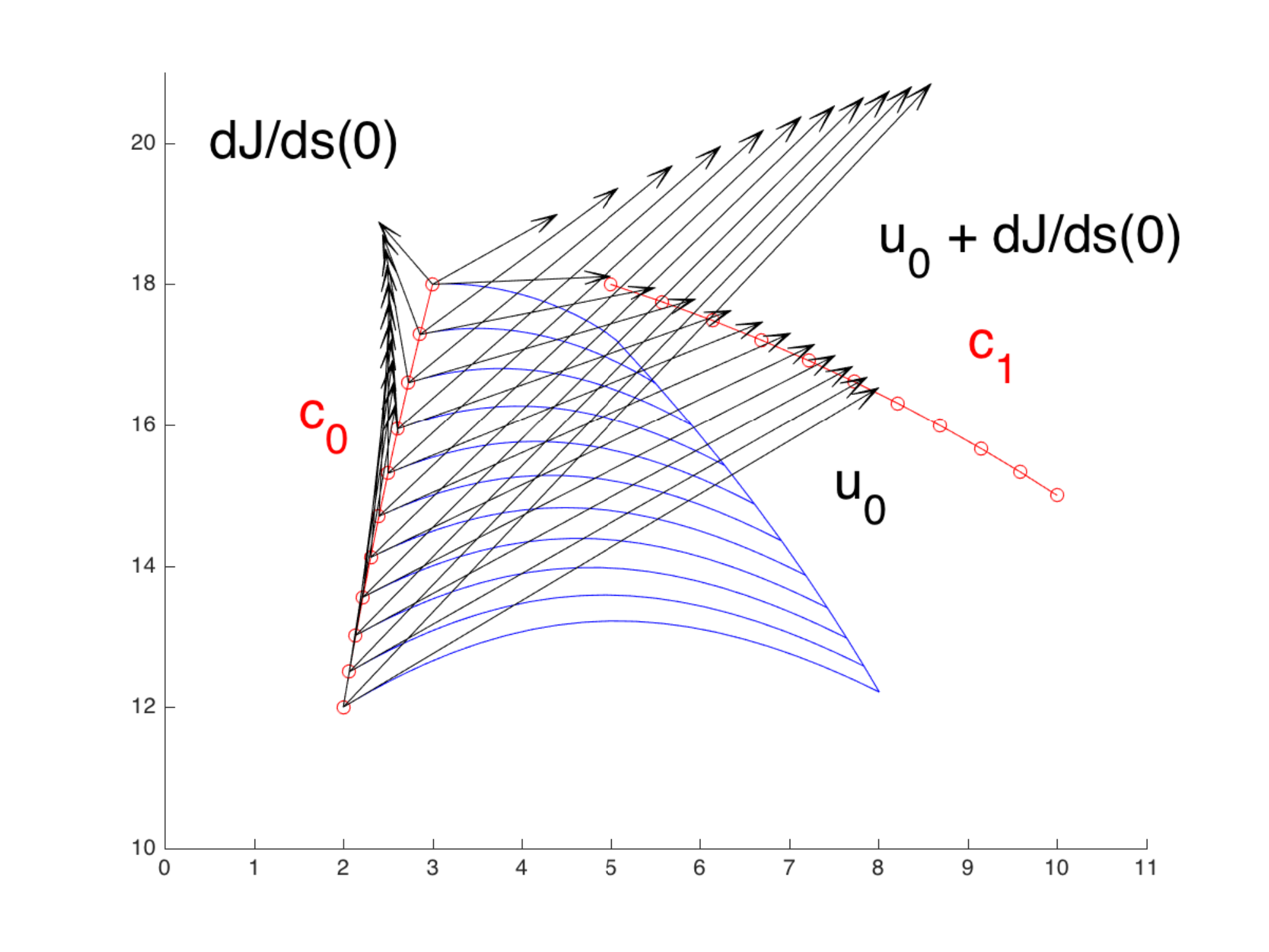}}
\subfloat{\includegraphics[width=0.32\textwidth]{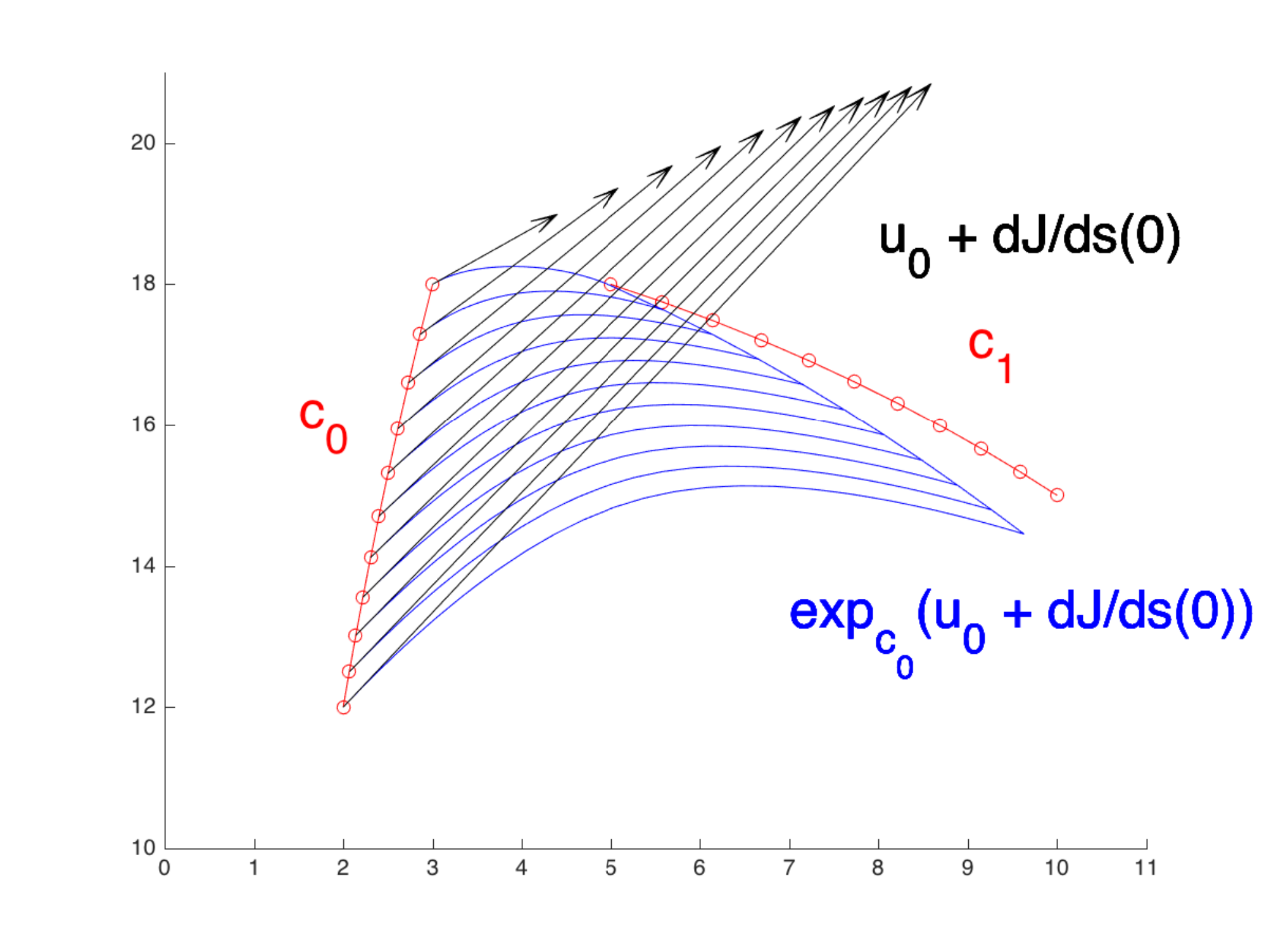}}\vspace{-0.7em}\\
\subfloat{\includegraphics[width=0.32\textwidth]{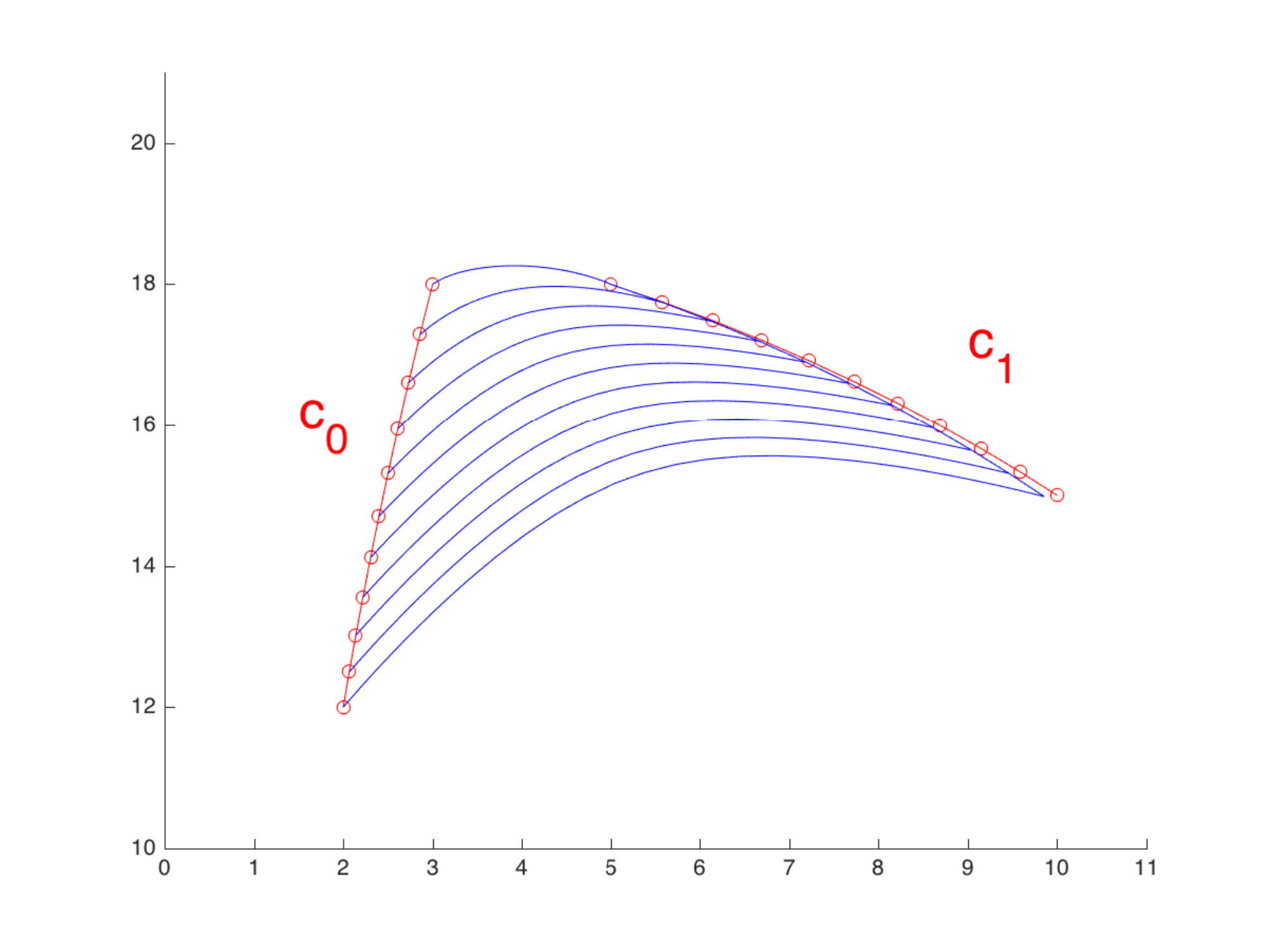}}
\subfloat{\includegraphics[width=0.32\textwidth]{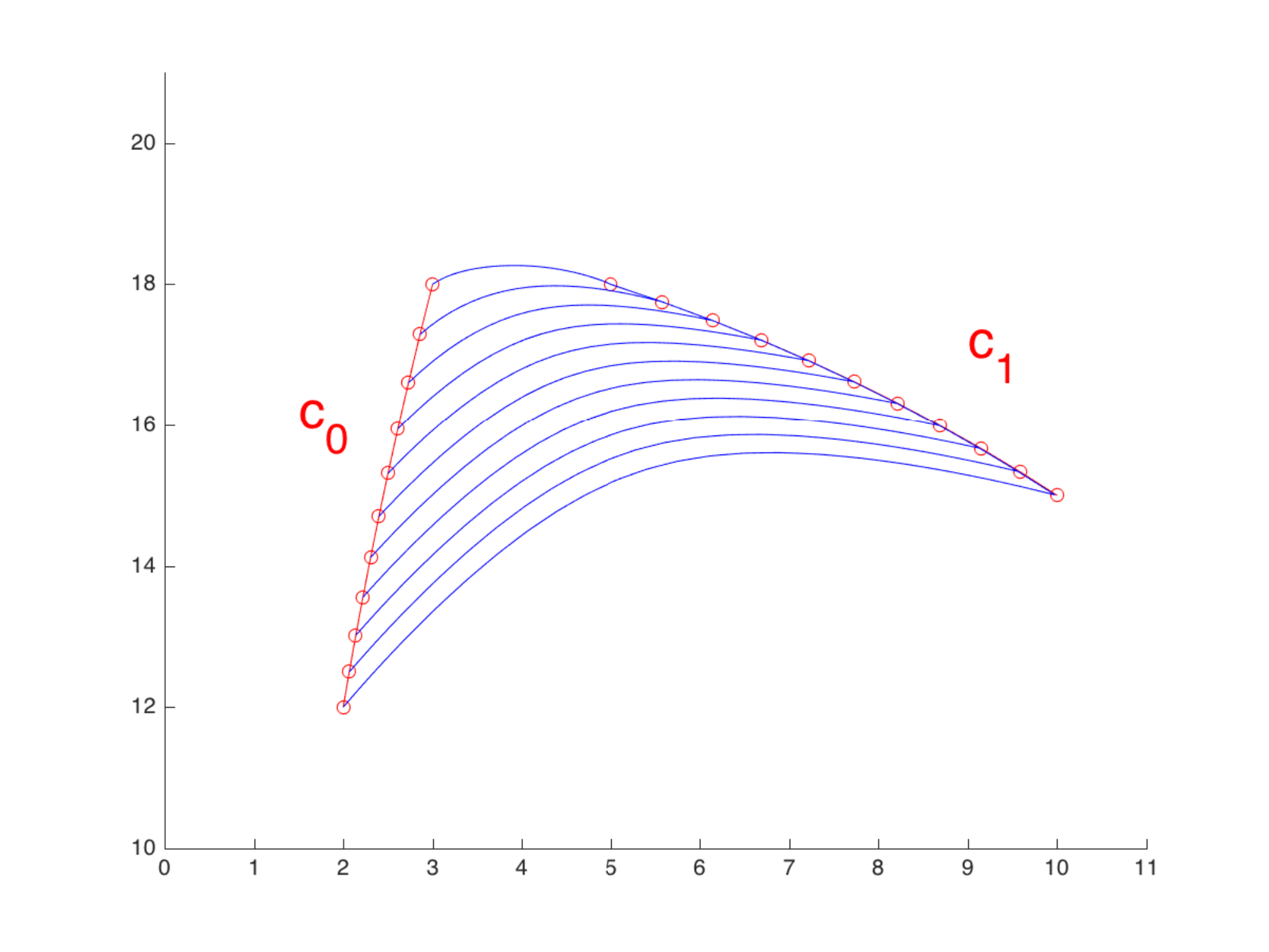}}
\caption{Geodesic shooting between two geodesics of the upper half-plane $\mathbb H$. The first six subfigures constitute the steps of the first iteration of the algorithm, and the last two show only the result of the two following iterations.}
\label{fig:toyex}
\end{figure}

Now, let us characterize the Jacobi fields of $\mathcal M$ to obtain the function $\phi$. A Jacobi field is a vector field that describes the way geodesics spread out on a manifold. Consider $a \mapsto c(a,\cdot,\cdot)$, $a \in (-\epsilon, \epsilon)$, a family of geodesics in $\mathcal M$, that is for each $a\in(-\epsilon, \epsilon)$, $[0,1] \ni s \mapsto c(a,s,\cdot)$ is a geodesic of $\mathcal M$. Then for all $a$, $c(a,\cdot,\cdot)$ verifies the geodesic equations
\begin{subequations} 
\begin{align}
& \nabla_sc_s(a,s,0) + r(a,s,0)  = 0 , \quad \forall s \label{jacobigeod1}\\
& \nabla_s\nabla_s q(a,s,t) + \left\|q(a,s,t) \right\| \left( r(a,s,t) + r(a,s,t)^T \right) = 0 , \quad \forall t,s, \label{jacobigeod2}
\end{align}
\end{subequations}
where $q = c_t/\sqrt{\left\| c_t \right\|}$ is the SRV representation of $c$ and $r$ is given by
\begin{equation*}
r(a,s,t) = \int_t^1 \mathcal R(q,\nabla_sq)c_s(a,s,\tau)^{\tau,t} \mathrm d\tau.
\end{equation*} 
Recall that we use the notation $w^T = \langle w,v(a,s,t) \rangle v(a,s,t)$ with $v = c_t/ \|c_t\|$ for the tangential component of a tangent vector $w \in T_{c(a,s,t)}M$. To characterize the way these geodesics spread out, we consider the Jacobi field $J : [0,1] \rightarrow T\mathcal M$,
\begin{equation*}
J(s,\cdot) = \left.\frac{\partial}{\partial a} \right|_{a=0}c(a,s,\cdot).
\end{equation*}
By decomposing $\nabla_s\nabla_sJ(s,0)$ and $\nabla_t\nabla_s\nabla_sJ(s,\tau)$ we can write the second order variation of $J$ as
\begin{equation}
\begin{split}
\nabla_s\nabla_s&J(s,t) \,\, = \,\,\, \left[ \left. \big(\nabla_a\nabla_s c_s + \mathcal R(c_s, J)c_s\big) \right|_{a=0,t=0} \right]^{0,t} \label{jacobi} \\
&+ \int_0^t \left[ \left.\big(\nabla_s\nabla_s\nabla_t J + \mathcal R(c_t,c_s)\nabla_sJ + \nabla_s \left(\mathcal R(c_t,c_s)J\right) \big)\right|_{a=0,t=\tau} \right ]^{\tau,t} \mathrm d\tau.
\end{split}
\end{equation}
The term $\nabla_s\nabla_s\nabla_tJ$ can be expressed as a function of $\nabla_s\nabla_s\nabla_aq$ by twice differentiating the equation $\nabla_aq = \|c_t\|^{-1/2}(\nabla_ac_t - \tfrac{1}{2}{\nabla_ac_t}^T)$ according to $s$. This gives
\begin{align*}
\nabla_s\nabla_s\nabla_aq &= \nabla_s\nabla_s(\|c_t\|^{-1/2})(\nabla_ac_t - \tfrac{1}{2} {\nabla_ac_t}^T) + 2 \nabla_s(\|c_t\|^{-1/2})\Big(\nabla_s\nabla_ac_t\\
&- \tfrac{1}{2} {\nabla_s\nabla_ac_t}^T -\tfrac{1}{2} \langle \nabla_ac_t, \nabla_sv\rangle v - \tfrac{1}{2} \langle \nabla_ac_t,v \rangle \nabla_sv\Big) + \|c_t\|^{-1/2}\Big( \nabla_s\nabla_s\nabla_ac_t\\
&- \tfrac{1}{2}{\nabla_s\nabla_s\nabla_ac_t}^T-\langle \nabla_s\nabla_ac_t,\nabla_sv\rangle v - \langle \nabla_s\nabla_ac_t,v\rangle \nabla_sv - \langle \nabla_ac_t,\nabla_sv\rangle \nabla_sv\\
&-\tfrac{1}{2} \langle \nabla_ac_t,\nabla_s\nabla_sv\rangle -\tfrac{1}{2}\langle \nabla_ac_t, v\rangle \nabla_s\nabla_sv\Big).
\end{align*}
Since $\nabla_ac_t = \nabla_tc_a = \nabla_aJ$ for $a=0$, we know that the term we are looking for is $\nabla_s\nabla_s\nabla_tJ = \nabla_s\nabla_s\nabla_ac_t$. Noticing that $W = Z - \tfrac{1}{2}Z^T$ is equivalent to $Z = W + W^T$ we get
\begin{equation}
\nabla_s\nabla_s\nabla_t J = W + W^T,
\end{equation} 
\begin{equation}
\label{w}
\begin{split}
W &= \big\langle \nabla_s\nabla_t J, \nabla_s v \big\rangle v + \big\langle \nabla_s\nabla_tJ, v \big\rangle \nabla_sv + \big\langle \nabla_tJ, \nabla_sv \big\rangle \nabla_sv \\
& + \tfrac{1}{2} \big\langle \nabla_tJ, \nabla_s\nabla_sv \big\rangle v + \tfrac{1}{2} \big\langle \nabla_tJ, v \big\rangle \nabla_s\nabla_s v + \sqrt{\left\| c_t \right\|} \Big[ \nabla_a\nabla_s\nabla_s q \\
&- 2 \nabla_s\big( \left\| c_t \right\|^{-1/2} \big) \Big( \nabla_s\nabla_tJ - \tfrac{1}{2} {\nabla_s\nabla_tJ}^T - \tfrac{1}{2} \langle \nabla_tJ, \nabla_sv\rangle v- \tfrac{1}{2} \langle\nabla_tJ,v\rangle \nabla_sv \Big)\\
& - \nabla_s\nabla_s\big( \left\| c_t \right\|^{-1/2} \big) \left( \nabla_t J - \tfrac{1}{2} \nabla_tJ^T \right) + \mathcal R(c_s,J)\nabla_sq + \nabla_s \left(\mathcal R(c_s,J)q\right) \Big].
\end{split}
\end{equation}
The terms $\nabla_a\nabla_sc_s(0,s,0)$ and $\nabla_a\nabla_s\nabla_sq(0,s,\tau)$ for all $\tau \in[0,1]$ can be obtained by differentiating the geodesic equations \eqref{jacobigeod1} and \eqref{jacobigeod2}
\begin{align*}
&\nabla_a\nabla_sc_s(0,s,0) + \nabla_a r(0,s,0) = 0, \quad \forall s \\
&\nabla_a\nabla_s\nabla_sq(0,s,t) + \nabla_a \left\| q(0,s,t) \right\| \left( r(0,s,t) + r(0,s,t)^T \right) \\
& \hspace{10em} + \left\| q(0,s,t) \right\| \left( \nabla_ar(0,s,t) + \nabla_a\left(r(0,s,t)^T\right) \right) = 0, \quad \forall t,s.
\end{align*}
The first one gives 
\begin{equation}
\label{nanscs}
\nabla_a\nabla_sc_s(s,0) = - \nabla_a r(s,0),
\end{equation}
and for all $s$ and $t$ we get
\begin{equation}
\label{nansnsq}
\begin{split}
\nabla_a\nabla_s\nabla_sq = - \sqrt{\|c_t\|} \Big( \nabla_ar +\nabla_ar^T \Big) - \frac{1}{\sqrt{\|c_t\|}} \bigg( \big\langle r,&\nabla_tJ\big\rangle v + \big\langle r, v \big\rangle \nabla_tJ \\
&+ \frac{1}{2} \big\langle \nabla_tJ,v \big\rangle \left( r - 3 r^T \right)\bigg).
\end{split}
\end{equation}
The only term left to compute is the variation $\nabla_ar$, which is by definition
\begin{equation*}
\nabla_ar(0,s,t) = \int_t^1 \nabla_aV_t(0,s,\tau) \, \mathrm d\tau,
\end{equation*}
if we define $V_t$ for any fixed $t$ by
\begin{equation*} 
V_t(a,s,\tau) = \left[ \mathcal R(q,\nabla_sq)c_s(a,s,\tau) \right]^{\tau,t}, \quad \tau\in [t,1].
\end{equation*}
Since the covariant derivative of $V_t$ in $\tau$ vanishes, we can write for any $t\leq \tau \leq 1$
\begin{equation*}
\nabla_aV_t(0,s,\tau) = \nabla_aV_t(0,s,t) + \int_t^\tau \left.\mathcal R(c_t,J)\right|_{t} \left(V_t(0,s,u)\right) \, \mathrm du.
\end{equation*}
Integrating this equation according to $\tau$ from $t$ to $1$ we obtain
\begin{equation}
\label{nar}
\nabla_ar(0,s,t) = (1-t) \nabla_aV_t(0,s,t) + \left.\mathcal R(c_t,J)\right|_{t} \left( \int_t^1 (1-\tau) V_t(0,s,\tau) \, \mathrm d\tau\right),
\end{equation}
where, since $V_t(0,s,t)=\mathcal R(q,\nabla_sq)c_s(0,s,t)$, we get for $\tau=t$
\begin{equation}
\label{navt}
\left.\nabla_aV_t\right|_t = \nabla_J\mathcal R(q,\nabla_sq)c_s + \mathcal R(\nabla_aq, \nabla_sq) c_s + \mathcal R(q,\nabla_a\nabla_sq)c_s + \mathcal R(q,\nabla_sq)\nabla_sJ,
\end{equation}
with finally
\begin{subequations}
\begin{align}
\nabla_aq &= \frac{1}{\sqrt{\|c_t\|}} \left( \nabla_tJ - \tfrac{1}{2}\nabla_tJ^T\right), \label{naq} \\
\nabla_a\nabla_sq &= \frac{1}{\sqrt{\|c_t\|}} \left( \nabla_s\nabla_tJ - \tfrac{1}{2} \nabla_s\nabla_tJ^T -\tfrac{1}{2} \big\langle \nabla_tJ,\nabla_sv\big\rangle v - \tfrac{1}{2} \big\langle \nabla_tJ,v\big\rangle \nabla_sv \right) \label{nansq} \\
& \qquad \qquad \qquad \qquad \,\,\,\,\, + \nabla_s\left( \|c_t\|^{-1/2} \right) \left( \nabla_tJ - \tfrac{1}{2} \nabla_tJ^T \right) + \mathcal R(J,c_s)q. \nonumber
\end{align}
\end{subequations}
\begin{figure}
\subfloat{\includegraphics[width=0.26\textwidth]{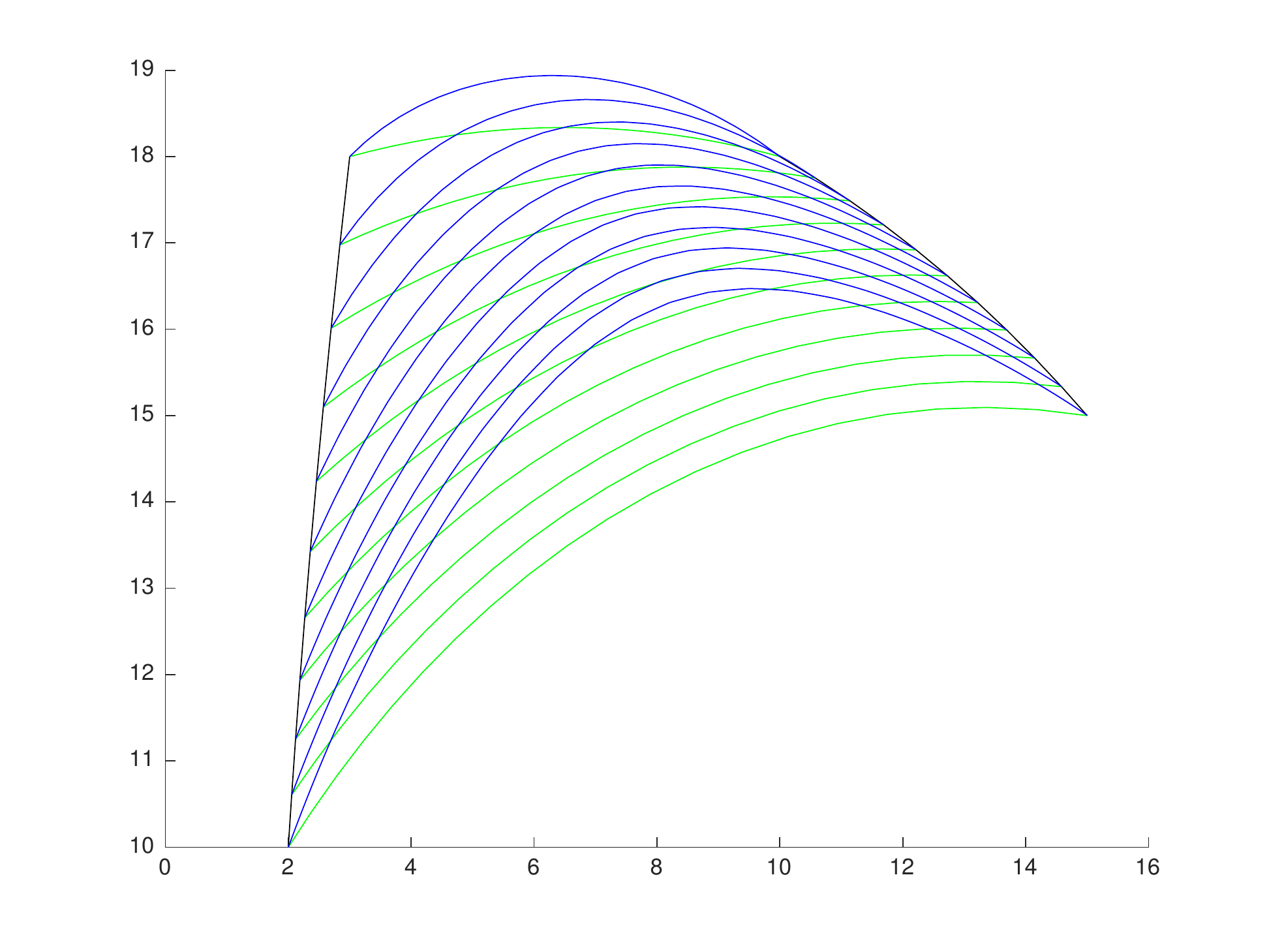}}\hspace{-1em}
\subfloat{\includegraphics[width=0.26\textwidth]{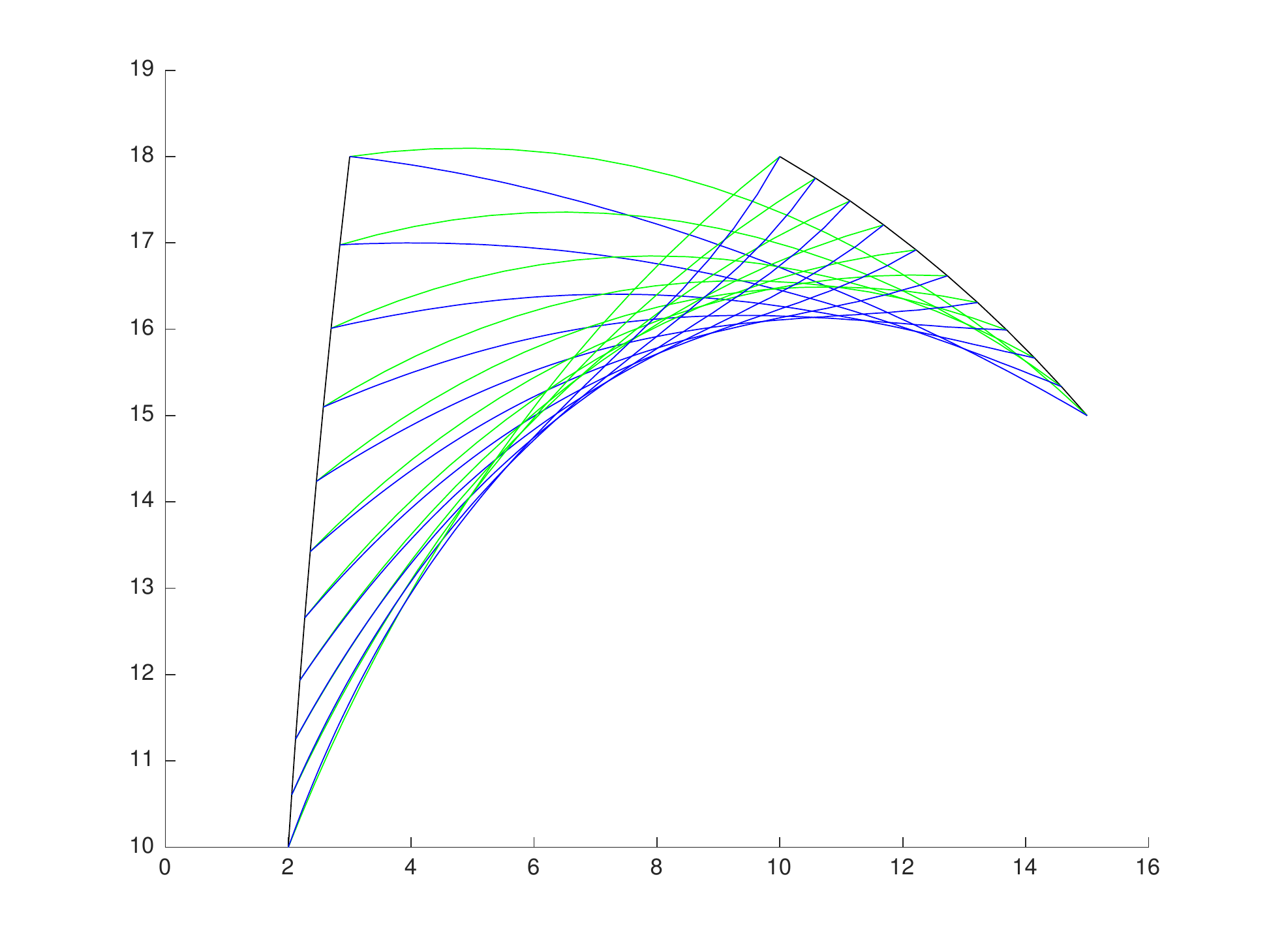}}\hspace{-1em}
\subfloat{\includegraphics[width=0.26\textwidth]{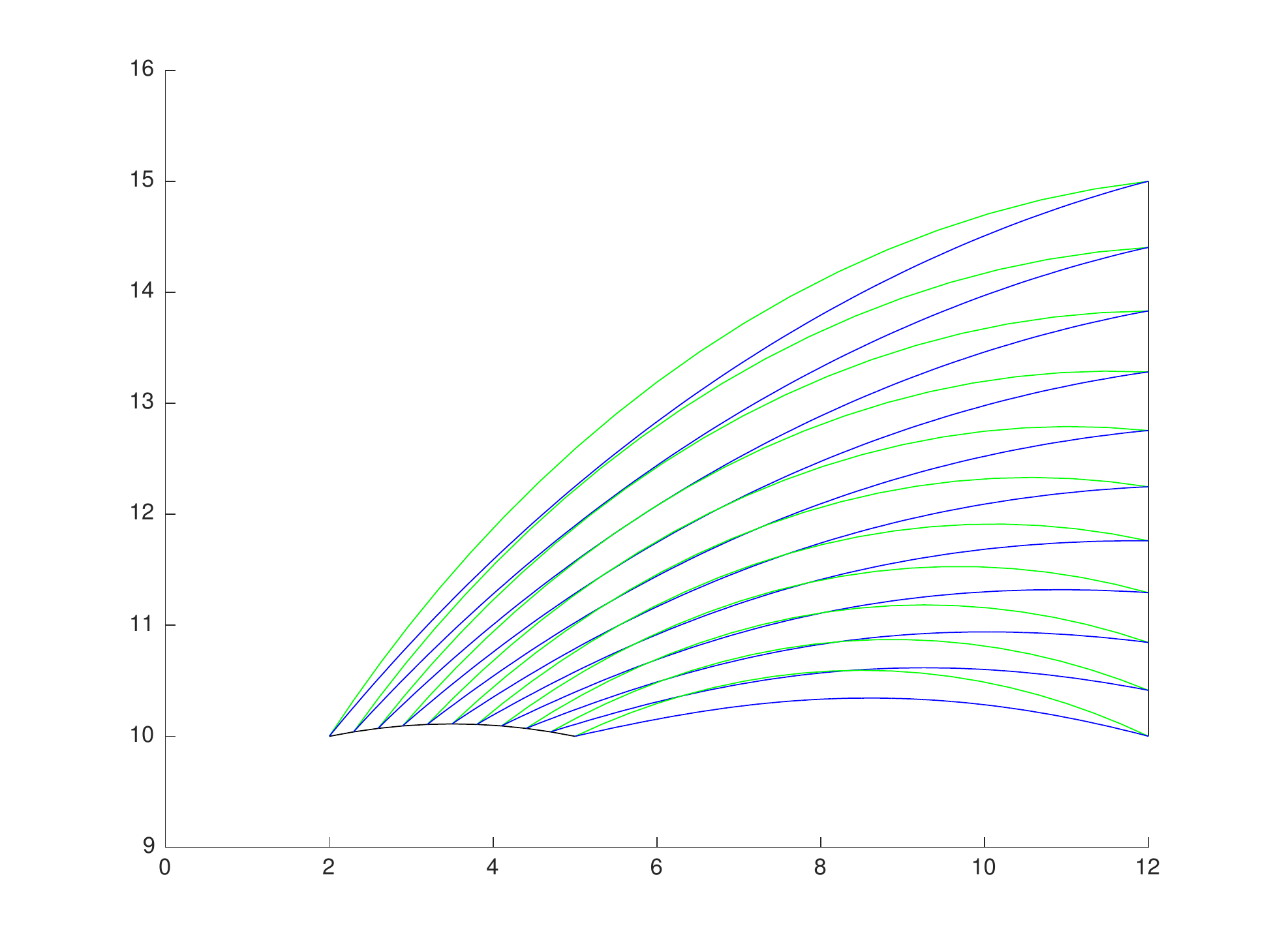}}\hspace{-1em}
\subfloat{\includegraphics[width=0.26\textwidth]{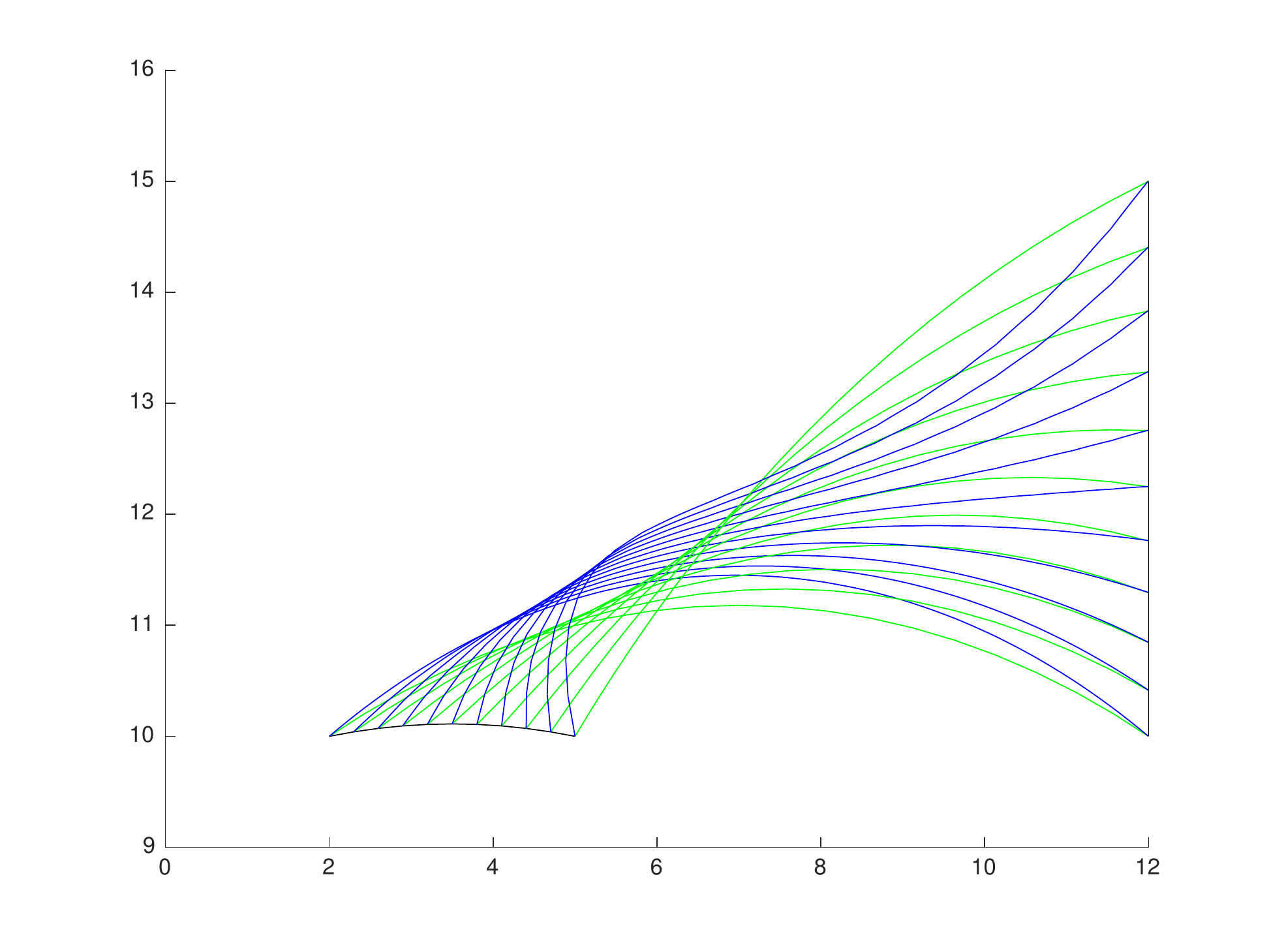}}
\caption{Optimal deformations between pairs of geodesics (in black) of the upper half-plane $\mathbb H$, for our metric (in blue) and for the $L^2$-metric (in green). The orientation of the right-hand curve is inverted in the second image compared to the first, and in the fourth compared to the third.}
\label{fig:toyex2}
\end{figure}

We can notice that, however complicated, the numbered equations \eqref{jacobi} to \eqref{nansq} when put together define a partial differential equation verified by the Jacobi field $J$. They allow us to iteratively compute $J(s+\epsilon,\cdot)$ and $\nabla_sJ(s+\epsilon,\cdot)$, for a fixed step $\epsilon >0$, knowing $J(s,\cdot)$ and $\nabla_sJ(s,\cdot)$. Indeed, we can estimate $\nabla_tJ(s,\cdot)$ since $J(s,t)$ is known for all $t$, as well as $\nabla_t\nabla_sJ(s,\cdot)$ since $\nabla_sJ(s,t)$ is known for all $t$, and finally $\nabla_s\nabla_tJ = \nabla_t\nabla_sJ + \mathcal R(c_s,c_t)J$. Assuming that we are able to compute the covariant derivative $\nabla_J\mathcal R$ of the curvature tensor, for example if we are in a symmetric space (then it is zero), we obtain an algorithm to compute the Jacobi fields in the space of curves. To summarize :
\begin{algorithm}[Jacobi fields in the space of curves in a symmetric space] 
\label{alg:jacobi}
\leavevmode\par \noindent
Input : $c$, $J_0$ and $W$. \\
Initialization : Set $J(0,t) = J_0(t)$ and $\nabla_sJ(0,t) = W(t)$ for $t\in[0,1]$.\\
Heredity : For $j=0,\hdots, m-1$, set $s=j\epsilon$ with $\epsilon = 1/m$ and
\begin{enumerate} 
\item for all $t$, set \vspace{-0.5em}
\begin{align*}
\nabla_tJ(s,t) &= \lim_{\delta\rightarrow 0}\frac{1}{\delta} \left( J(s,t+\delta)^{t+\delta,t} - J(s,t) \right), \\
\nabla_t\nabla_sJ(s,t) &= \lim_{\delta\rightarrow 0}\frac{1}{\delta} \left( \nabla_sJ(s,t+\delta)^{t+\delta,t} - \nabla_sJ(s,t) \right), \\[5pt]
\nabla_s\nabla_tJ(s,t) &= \nabla_t\nabla_sJ(s,t) + \mathcal R(c_s,c_t)J(s,t).
\end{align*}
\item Compute $r(s,t) = \int_t^1 \mathcal R(q,\nabla_sq)c_s(s,\tau)^{\tau,t} \mathrm d\tau$ for all $t\in[0,1]$.
\vspace{0.5em}
\item Compute $\nabla_aq(0,s,t)$, $\nabla_s\nabla_aq(0,s,t)$ and $\nabla_aV_{t}(0,s,t)$ for all $t$ using \eqref{naq}, \eqref{nansq} and \eqref{navt}, and deduce $\nabla_ar(0,s,t)$ using Equation \eqref{nar}.
\vspace{0.5em}
\item Compute $\nabla_a\nabla_sc_s(0,s,0)$ and $\nabla_a\nabla_s\nabla_sq(0,s,t)$ for all $t$ using equations \eqref{nanscs} and \eqref{nansnsq}.
\vspace{0.5em}
\item Compute $W(s,t)$ using \eqref{w} and $\nabla_s\nabla_s\nabla_t J(s,t) = W(s,t) + W^T(s,t)$, and deduce $\nabla_s\nabla_sJ(s,t)$ for all $t$ using Equation \eqref{jacobi}.
\vspace{0.5em}
\item Finally, for all $t\in[0,1]$, set
\begin{align*}
J(s+\epsilon,t) &= \left[ \,\,J(s,t) + \epsilon \nabla_sJ(s,t) \,\,\right]^{s,s+\epsilon}, \\
\nabla_sJ(s+\epsilon,t) &= \left[ \,\,\nabla_sJ(s,t) + \epsilon \nabla_s\nabla_sJ(s,t) \,\,\right]^{s,s+\epsilon}.
\end{align*}
\end{enumerate}
Output : $J(1)$.
\end{algorithm}
Using a discretization of Algorithms \ref{alg:expmap}, \ref{alg:geodshoot} and \ref{alg:jacobi}, we are able to compute an approximation of the optimal deformation between two curves, as shown in a toy example in Figure \ref{fig:toyex}. In this simple case we perform geodesic shooting between two geodesics $c_0$ and $c_1$ (in red) of the hyperbolic half-plane $\mathbb H$. The reasons for our interest in this particular space, as well as the tools needed to work in it, are given in the next section. The first two lines of Figure \ref{fig:toyex} show the different steps of the first iteration of geodesic shooting, and the last line gives only the last step of the following two iterations. In this simple case, we can see that we converge in only three iterations. Further toy examples are given in Figure \ref{fig:toyex2}, where we show the optimal deformations between pairs of geodesics of $\mathbb H$ (in blue), compared to the $L^2$-geodesics (in green). We can see in the first image that our metric has a tendency to "shrink" the curves in the center of the deformation compared to the $L^2$-metric. We also show the influence of the orientation of the curves, on which the deformations depend. We do not give the details of the discretization used for these examples or the ones in the following section here. A detailed description of this discrete model will be given in a forthcoming paper.

\section{Example : curves in the hyperbolic half-plane $\mathbb H$}

In this section we consider the case where the base manifold $M$ is a symmetric manifold of negative curvature, the two-dimensional hyperbolic space $\mathbb H$. We first explain why this space can be interesting for applications, namely as it coincides with the statistical manifold of Gaussian densities equipped with the Fisher Information metric. Then we give some basic tools -- exponential map, logarithm map, curvature tensor -- needed to implement the previous algorithms in $\mathbb H$. Finally, we consider a specific application of curve analysis in that space, for the statistical study of locally stationary radar signals. We explain how this framework gives us curves lying in the hyperbolic plane, and we present some simulation results. 

\subsection{The hyperbolic half-plane as a statistical manifold}
It is possible to adopt a geometrical point of view to solve problems in various fields such as statistical inference, information theory or signal processing \cite{frech}, \cite{burbea}, \cite{amari}. This framework is given by information geometry. Each element of a parametric family of probability densities $\{ f(\cdot,\theta), \theta \in \Theta\}$ can be seen as a point in the manifold of parameters $\Theta$. Intuitively, it is easy to see that the Euclidean metric is not always appropriate to compare two probability distributions in that space. For example, each univariate Gaussian distribution can be identified with its mean and standard deviation $(m,\sigma)$ in the upper half-plane $\mathbb R \times \mathbb R^*_+$. As explained in \cite{costa}, two univariate Gaussian densities $\mathcal N(m_1,\sigma_1)$ and $\mathcal N(m_2,\sigma_1)$ with different means but the same standard deviation get "closer" to each other as their common standard deviation increases, meaning that intuitively the distance between the points of coordinates $(m_1,\sigma_1)$ and $(m_2,\sigma_1)$ in the upper half-plane should be greater than the distance between the points $(m_1,\sigma_2)$ and $(m_2,\sigma_2)$ for $\sigma_2>\sigma_1$. A more pertinent Riemannian structure on the space of parameters $\Theta$ is the one induced by the Fisher information metric, defined in its matrix form as the Fisher information. If the parameter $\theta \in \mathbb R^d$ is $d$-dimensional and $\mathbb E$ denotes the expected value,
\begin{equation*}
g_{ij}(\theta) = I(\theta)_{ij}=\mathbb E\left[ \left( \frac{\partial}{\partial \theta_i} \ln f(X;\theta)\right) \left(  \frac{\partial}{\partial \theta_j} \ln f(X;\theta)\right) \right],
\end{equation*}
for any $1\leq i,j \leq d$ and $\theta = (\theta_1, \hdots, \theta_d)$. This metric is chosen, among other reasons, because it has statistical meaning : in parameter estimation, the Fisher information measures the "amount of information" on the parameter $\theta$ contained in data sampled from the density $f(\cdot,\theta)$; it also gives a fundamental limit to the precision at which one can estimate this $\theta$, in the form of the Cramer-Rao bound. In the case of univariate Gaussian densities $\mathcal N(m,\sigma^2)$, Fisher geometry amounts to hyperbolic geometry. More precisely, the space of parameters $(m,\sigma)$ equipped with the Fisher Information metric is in bijection with the hyperbolic half-plane via the change of variables $(m,\sigma) \mapsto (\frac{m}{\sqrt{2}},\sigma)$. Indeed, with this rescaling of the mean, the Fisher Information matrix becomes
\begin{equation*}
g_{ij}(m,\sigma) = \left[ \begin{matrix} \frac{1}{\sigma^2} & 0 \\ 0 & \frac{1}{\sigma^2} \end{matrix} \right],
\end{equation*}
which defines the Riemannian metric of the well-known hyperbolic half-plane. This is coherent with the example given above, since in the hyperbolic half-plane the distance between the points of coordinates $(m_1,\sigma)$ and $(m_2,\sigma)$ decreases as $\sigma$ increases for fixed values of $m_1$, $m_2$. The differential geometry of Gaussians has proved useful for applications, e.g. in image processing where in the image model, each pixel is represented by a univariate Gaussian distribution \cite{ang}, and in radar signal processing \cite{arn}, \cite{barb3}, \cite{pilte}, \cite{lb2}, as we will see in Section \ref{spectralestimation}.

\subsection{Geometry of the hyperbolic half-plane}
First, let us give a few tools which are necessary to work in the hyperbolic half-plane representation. Along with the Poincar\'e disk, the Klein model and others, the hyperbolic half-plane $\mathbb H = \{z=x+iy \in \mathbb C, \,y >0 \}$ is one of the representations of two-dimensional hyperbolic geometry. The Riemannian metric is given by
\begin{equation*}
ds^2 = \frac{dx^2 + dy^2}{y^2}.
\end{equation*}
This means that the scalar product between two tangent vectors $u=u_1 + i u_2$ and $v=v_1+iv_2$ at a point $z=x+iy$ is
\begin{equation*}
\langle u,v \rangle = \frac{u_1v_1+u_2v_2}{y^2}.
\end{equation*}
Using the usual formula (see e.g. \cite{doc}) to compute the Christoffel symbols, we can easily compute the covariant derivative of a vector field $v(t)=v_1(t)+iv_2(t)$ along a curve $c(t)=x(t)+iy(t)$ in $\mathbb H$. It is given by $\nabla_{\dot c(t)} v = X(t) + i Y(t)$ where
\begin{equation}
\label{cov}
X = \dot{v_1} - \frac{\dot x v_2 + \dot y v_1}{y}, \quad Y = \dot{v_2} + \frac{ \dot x v_1 - \dot y v_2}{y}.
\end{equation}
Let us now remind a well-known expression \cite{doc} for the Riemann curvature tensor in a manifold of constant sectional curvature. Recall that $\mathbb H$ has constant sectional curvature $K=-1$.
\begin{figure}
\centering
\includegraphics[width=10cm]{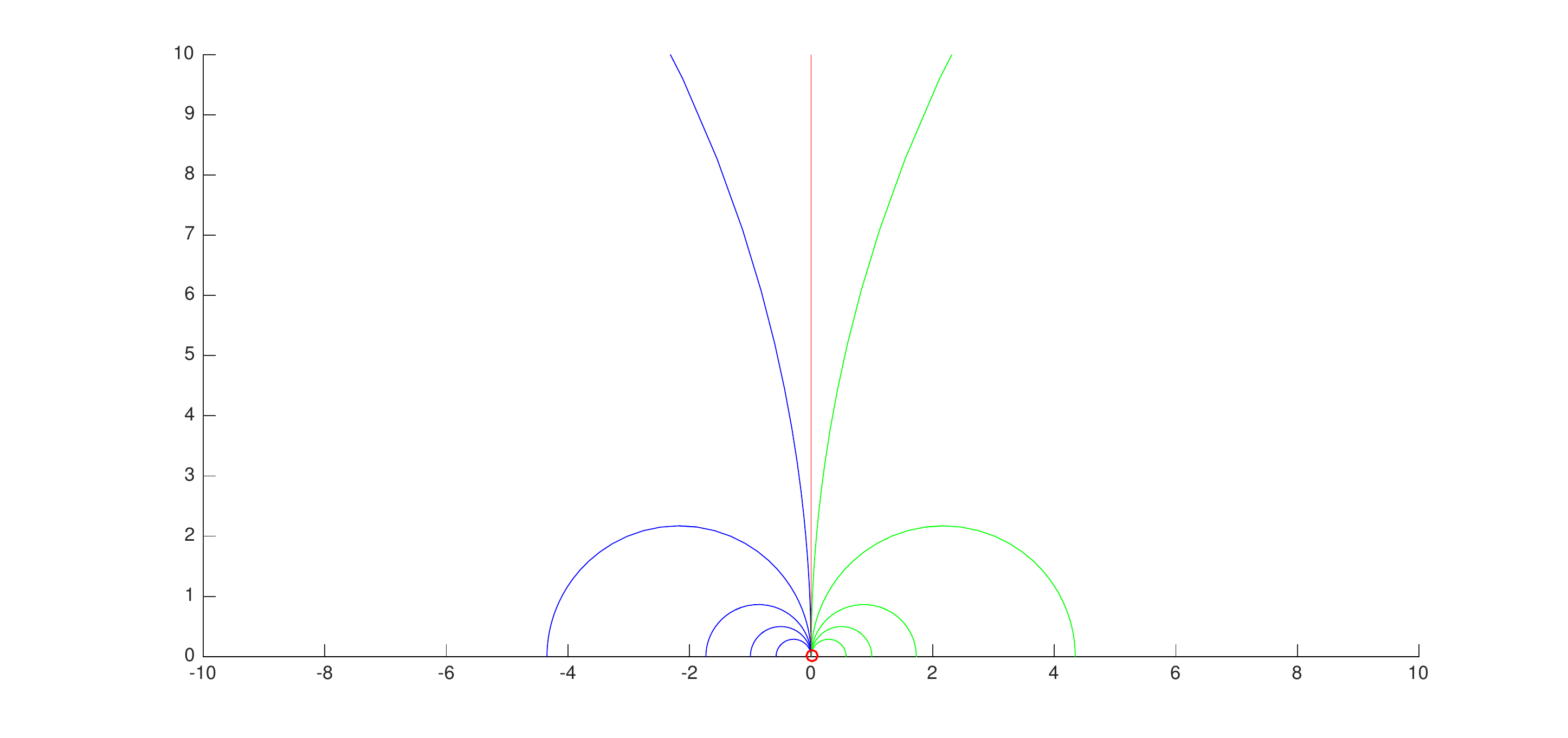}
\caption{Geodesics of the hyperbolic half-plane.}
\label{fig:hyp}
\end{figure}
\begin{proposition}[Curvature tensor] 
Let $X, Y, Z$ be three vector fields on a manifold of constant sectional curvature $K$. The Riemann curvature tensor can be written
\begin{equation*}
\mathcal R(X,Y)Z = K \left( \langle Y, Z \rangle X - \langle X, Z \rangle Y \right).
\end{equation*}
\end{proposition}
For the algorithms described above, we need to be able to compute the geodesic starting from a point $p\in\mathbb{H}$ at speed $u_0\in T_p\mathbb{H}$ -- in other words, the exponential map $u \mapsto \exp^{\mathbb H}_p(u)$ -- as well as the geodesic linking two points $p$ and $q$, with the associated initial vector speed -- the inverse $q \mapsto \log^{\mathbb H}_p(q)$ of the exponential map. The geodesics of the hyperbolic half-plane are vertical segments and half-circles whose origins are on the x-axis, as shown in Figure \ref{fig:hyp}, and they can be obtained as images of the vertical geodesic following the y-axis by a Moebius transformation $z \mapsto \frac{az+b}{cz+d}$, with $ad-bc = 1$. To be complete, we give the proofs of the three following propositions in the appendix.
\begin{proposition}[Geodesics of $\mathbb H$ and logarithm map]
\label{prop:geodH}
Let $z_0 = x_0 + i y_0$ and $z_1 = x_1 + i y_1$ be two elements of $\mathbb H$. 
\begin{itemize}
\item If $x_0=x_1$, then the geodesic going from $z_0$ to $z_1$ is the segment $\gamma(t)=i y(t)$ with $y(t)=y_0 e^{t\ln\frac{y_1}{y_0}}$, and the logarithm map is given by 
\begin{equation*}
\log^{\mathbb H}_{z_0}(z_1)= i y_0 \ln\frac{y_1}{y_0}.
\end{equation*}
\item If $x_0 \neq x_1$, the geodesic is given by $\gamma(t) = x(t) + i y(t)$ with
\begin{equation*}
x(t) = \dfrac{ bd + ac \bar y(t)^2}{ d^2 + c^2 \bar y(t)^2}, \quad 
y(t) = \dfrac{ \bar y(t)}{ d^2 + c^2 \bar y(t)^2}, \quad t\in[0,1],
\end{equation*}
where the coefficients of the Moebius transformation can be deduced from the center $x_\Omega$ and the radius $R$ of the semi-circle going through $z_0$ and $z_1$: $a = \frac{1}{2} \left( \frac{x_\Omega}{R} +1\right)$, $b = x_\Omega - R$, $c = \frac{1}{2R}$, $d = 1$, and for all $t\in[0,1]$, 
\begin{equation*}
\bar y(t) = \bar{y_0} e^{Kt}, \quad \text{with } K = \ln\frac{\bar{y_1}}{\bar{y_0}}, \,\,\, \bar{y_0} = -i \frac{a z_0 + b}{c z_0 +d} \, \text{ and } \bar{y_1} = -i \frac{a z_1+ b}{c z_1 +d}. 
\end{equation*}
The logarithm map is in turn given by 
\begin{equation*}
\log^{\mathbb H}_{z_0}(z_1) = \frac{ 2 cd K \bar{y_0}^2 }{ ( d^2 + c^2 \bar{y_0}^2 )^2 } + i \frac{ K \bar{y_0} (d^2 - c^2 \bar{y_0}^2) }{ (d^2 + c^2 \bar{y_0}^2 )^2 }.
\end{equation*}
\end{itemize}
\end{proposition}
We now give the exponential map in $\mathbb H$.
\begin{proposition}[Exponential map in $\mathbb H$]
\label{prop:expH}
Let $z_0 = x_0 + i y_0$ be an element of $\mathbb H$ and $u_0 = \dot{x_0} + i \dot{y_0}$ a tangent vector. Then the exponential map is given by $\exp^{\mathbb H}_{z_0}(u_0) = \gamma(1)$, where
\begin{itemize}
\item if $\dot{x_0} = 0$, $\gamma(t) = i y_0 e^{t \frac{\dot{y_0}}{y_0}}$, \vspace{0.5em}
\item if $\dot{x_0} \neq 0$, $\gamma(t) = x(t) + i y(t)$ with
\begin{equation*}
x(t) = \dfrac{ bd + ac \bar y(t)^2}{ d^2 + c^2 \bar y(t)^2}, \quad
y(t) = \dfrac{ \bar y(t)}{ d^2 + c^2 \bar y(t)^2}, \quad t\in[0,1].
\end{equation*}
The coefficients $a, b, c, d$ of the Moebius transformation can be computed as previously from the center $x_\Omega=x_0 + y_0 \frac{\dot{y_0}}{\dot{x_0}}$ and the radius $R=\sqrt{(x_0-x_\Omega)^2 + y_0^2}$ of the semi-circle of the geodesic, and for all $t\in[0,1]$,
\begin{equation*}
\bar y(t) = \bar{y_0} e^{t\frac{\dot{\bar{y_0}}}{\bar{y_0}}}, \quad \text{with } \bar{y_0} = -i \frac{a z_0 + b}{c z_0 +d} \, \text{ and } \dot{\bar{y_0}}=\frac{ \dot{x_0} (d^2 + c^2 \bar{y_0}^2)^2 }{ 2 cd \bar{y_0} }. 
\end{equation*}
\end{itemize}
\end{proposition}
Finally, we give the expression of parallel transport along a geodesic in the hyperbolic plane. 
\begin{proposition}[Parallel transport in $\mathbb H$]
\label{prop:ptH}
Let $t \mapsto \gamma(t)$ be a curve in $\mathbb{H}$ with coordinates $x(t)$, $y(t)$, and $u_0\in T_{\gamma(t_0)}\mathbb H$ a tangent vector. The parallel transport of $u_0$ along $\gamma$ from $t_0$ to $t$ is given by
\begin{equation*}
u(t) = \frac{y(t)}{y(t_0)} \left( \begin{matrix}
\,\,\,\, \cos \theta(t_0,t) & \sin \theta(t_0,t) \\
-\sin \theta(t_0,t) & \cos \theta(t_0,t) \end{matrix} \right) u_0,
\end{equation*}
where $\theta(t_i,t_f) = \int_{t_i}^{t_f} \frac{\dot x(\tau)}{y(\tau)} \mathrm d\tau$. If $\gamma$ is a vertical segment then $\theta(t_i,t_f)=0$, and if it is a portion of a circle, we get
\begin{equation*}
\theta(t_i, t_f) = 2 \left( \arg(d + ic \bar y(t_f)) - \arg(d + ic \bar y(t_i)) \right),
\end{equation*}
where the coefficients $c$ and $d$ of the Moebius transformation can be computed as explained previously, and $\bar \gamma = i \bar y$ is the pre-image of $\gamma$ by that transformation.
\end{proposition}
Now that we have these explicit formulas at our disposal, we are able to test the algorithms described above in the simple case where the base manifold $M$ has constant sectional curvature $K=-1$. Note that computations are further simplified by the existence of a global chart.

\subsection{Spectral estimation of locally stationary radar signals}
\label{spectralestimation}
In radar signal processing, given an observation of a signal, it is useful to estimate the spectrum of the underlying process, as it is indicative of its structure. If we are interested in the temporal modulations of that signal, we can estimate several spectra for that same signal and study their evolution in time. The study of these time-frequency spectra, or spectrograms, is at the heart of micro-Doppler analysis. Here we explain how a time series of spectra can be represented as a curve in the Poincar\'e polydisk.

The data we use for this example is synthetic data generated by a simulator of helicopter signatures. Using this simulator, we obtain a series $z=(z_1,\hdots ,z_N)\in \mathbb C^N$ of $N$ complex numbers that simulates the reflected signal received by a fixed radar antenna after sending a burst of $N$ pulses in the direction of a fixed helicopter. Given this vector of $N$ observations, the goal is to study the temporal evolutions of the underlying process. To do so, we consider that this process is locally stationary and Gaussian, and we estimate a spectrum for each stationary portion. More precisely, using a gliding window of size $n<N$ to be adjusted, we estimate a high resolution spectrum for each position of the window of size $n$ on the vector of size $N$. This gives us a time series of $N-n+1$ spectra, which we index by $1\leq i \leq N-n+1$.

Let us now explain how we estimate and represent these spectra, each of which corresponds to the observation $z^{(i)}=(z_i,\hdots ,z_{i+n-1})$ of a centered stationary Gaussian time series $Z^{(i)}$. To overcome the low resolution issues of the classical FFT-based spectral estimation methods, Burg suggested in the 1970s an alternative method based on autoregressive processes \cite{burg}. Given the partial knowledge of the autocorrelation function of a stationary and Gaussian process $Z$, Burg showed that the process which maximizes the entropy -- that is, which adds the fewest assumptions on the data -- is an autoregressive process of the appropriate order. Following this maximum entropy approach, we estimate an autoregressive spectrum for each stationary portion $z^{(i)}$ using the so-called Burg algorithm, see e.g. \cite{arn}. Burg also showed that the second order statistics of such a process $Z$ can be equivalently represented by its covariance matrix $\Sigma_n = \mathbb E(ZZ^*)$, a Toeplitz (because of the stationarity) Hermitian Positive Definite (THPD) matrix of size $n$, or the so-called \emph{reflection coefficients} $(P_0, \mu_1, \hdots, \mu_n)$ of the autoregressive model, as there exists a bijection \cite{tren}, \cite{ver}
$\Psi : \mathcal T_n \rightarrow \mathbb R_+^* \times D^{n-1}$,
\begin{equation*}
\Psi : \Sigma_n \mapsto (P_0,\mu_1,\hdots,\mu_{n-1}),
\end{equation*}
between the space $\mathcal T_n$ of THPD matrices of size $n$ and the product space $\mathbb R_+^* \times D^{n-1}$ where these reflection coefficients live. Here $D = \left\{ z \in \mathbb C \,|\, |z| <1 \right\}$ is the unit disk of the complex plane. This means that each stationary portion $z^{(i)}$ of size $n$ (the size of the gliding window) can be equivalently parameterized by its covariance matrix in $\mathcal T_n$ or by an element of the product space $\mathbb R_+^* \times D^{n-1}$. 

We choose to work with the latter representation, because we can select a convenient metric on that space. Indeed, the Legendre dual of the Fisher Information metric \cite{barb3}, defined on $\mathcal T_n$ as the hessian of minus the entropy,
\begin{equation*}
\Phi(\Sigma_n) = - \ln( \det \Sigma_n)-n\ln(2\pi e), \quad \Sigma_n \in \mathcal T_n,
\end{equation*}
has a nice expression in $\mathbb R_+^* \times D^{n-1}$, in which the Riemannian metric of the Poincar\'e disk appears \cite{barb1}
\begin{equation*}
ds^2 = n \left(\frac{dP_0}{P_0}\right)^2 + \sum_{k=1}^{n-1} (n-k) \frac{\left| d\mu_k \right|^2}{(1-\left|\mu_k\right|^2)^2},
\end{equation*}
with $(P_0,\mu_1,\hdots,\mu_{n-1})=\Psi(\Sigma_n)$. In other words, equipped with the Legendre dual of the Fisher information metric, the space of reflection coefficients becomes the product manifold $\mathbb R_+^* \times \mathbb D^{n-1}$, where $\mathbb D$ is the Poincar\'e disk. This means that each stationary portion $Z^{(i)}$ of our radar signal can be parameterized in the product manifold $\mathbb R_+^* \times \mathbb D^{n-1}$ by a set of coefficients $(P_0(i), \mu_1(i), \hdots, \mu_{n-1}(i))$, and so the entire locally stationary radar signal is represented by a time series in that space, which corresponds to a set of observations of the "real" evolution $(P_0(t), \mu_1(t), \hdots, \mu_{n-1}(t))$ of the locally stationary signal. 
\begin{figure}
\centering
\subfloat{\includegraphics[width=0.45\textwidth]{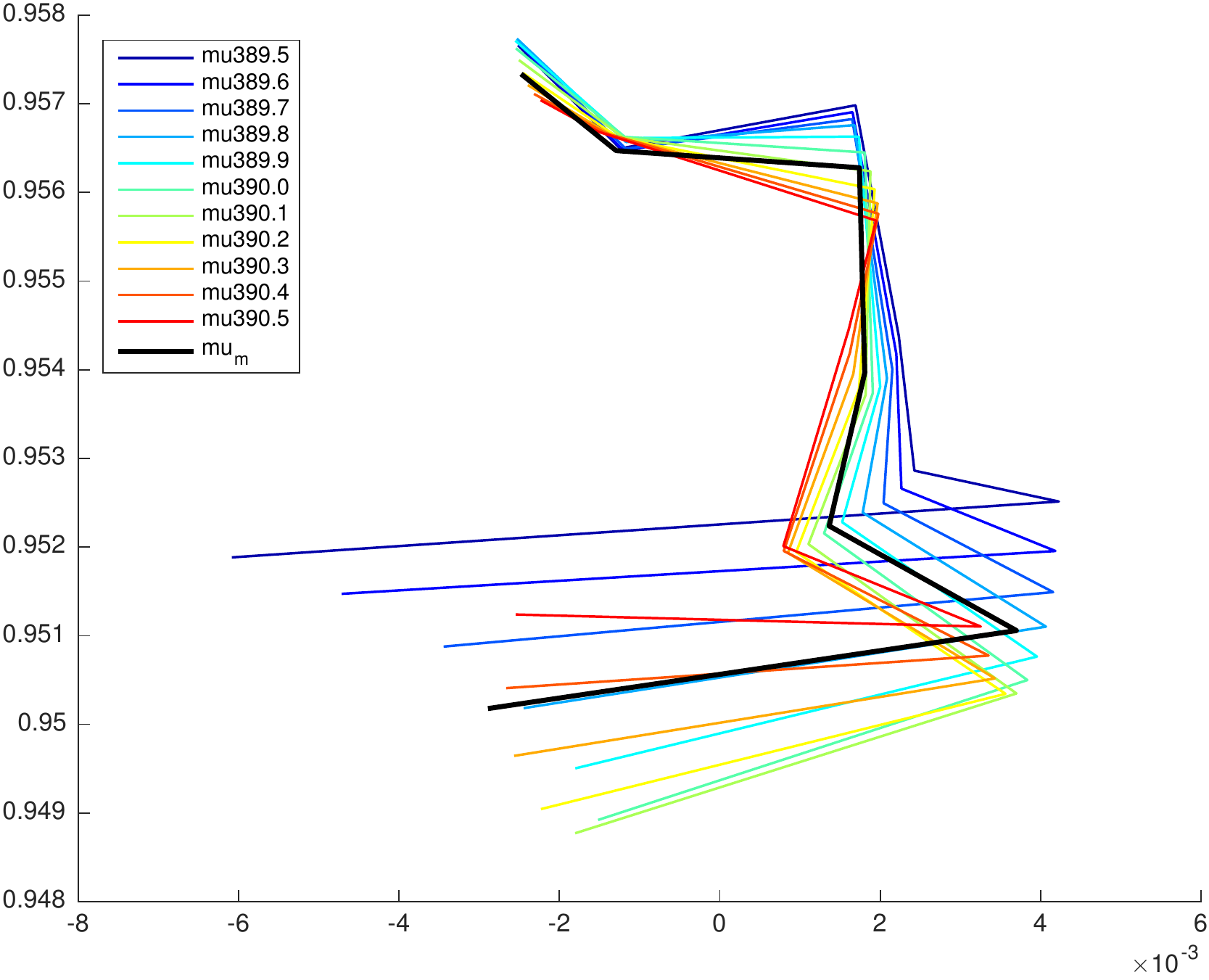}}
\subfloat{\includegraphics[width=0.45\textwidth]{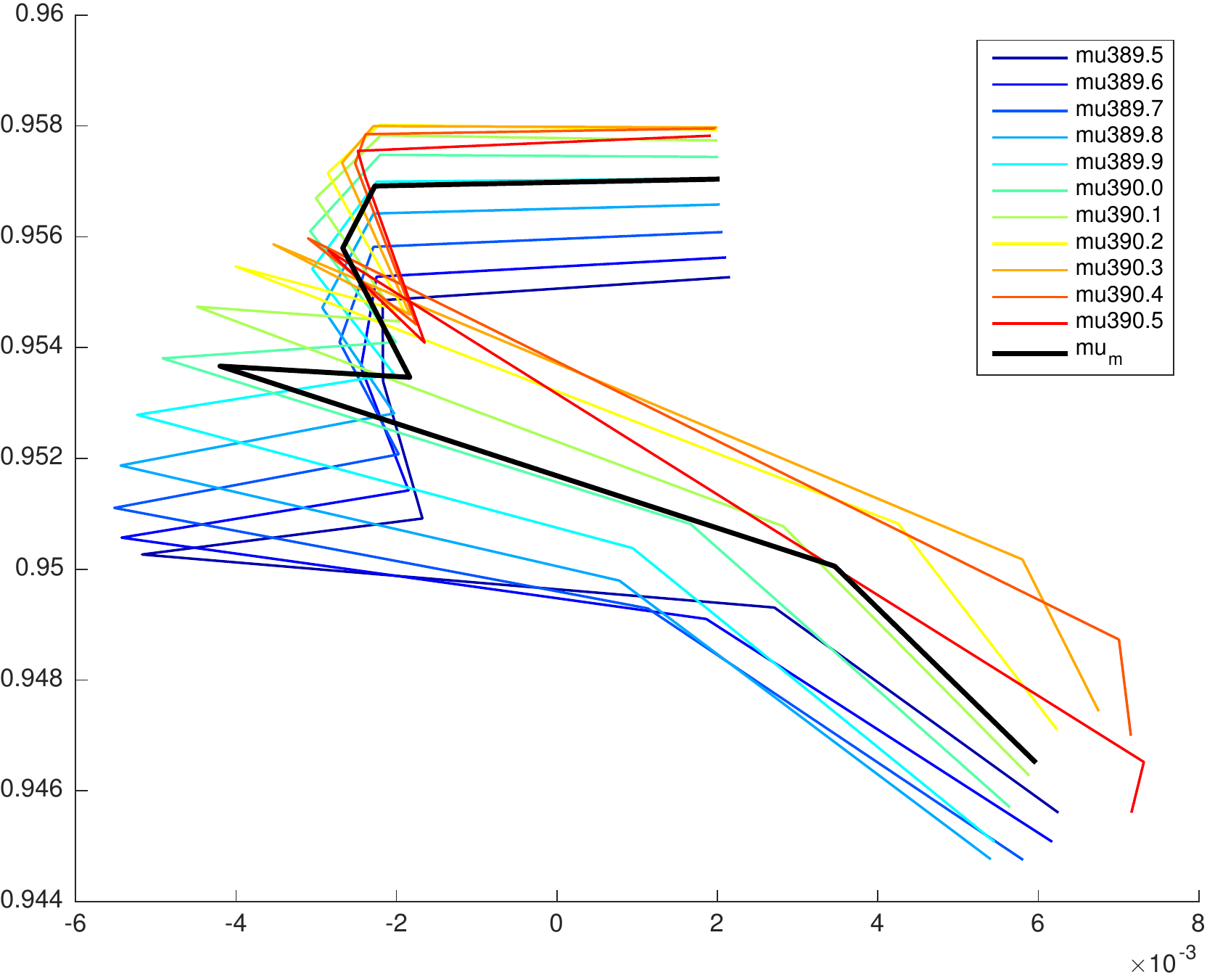}}\\ 
\subfloat{\includegraphics[width=0.45\textwidth]{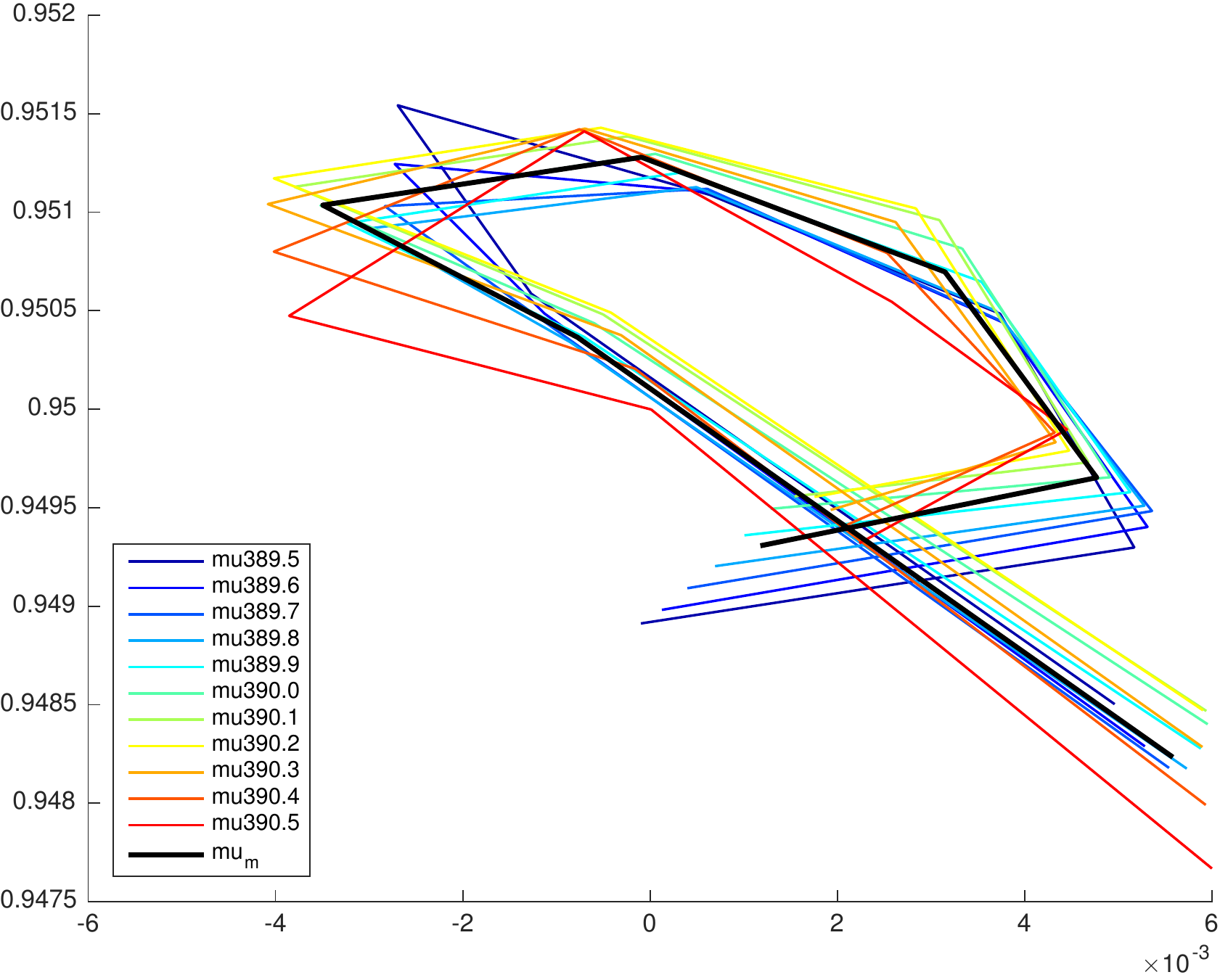}}
\subfloat{\includegraphics[width=0.45\textwidth]{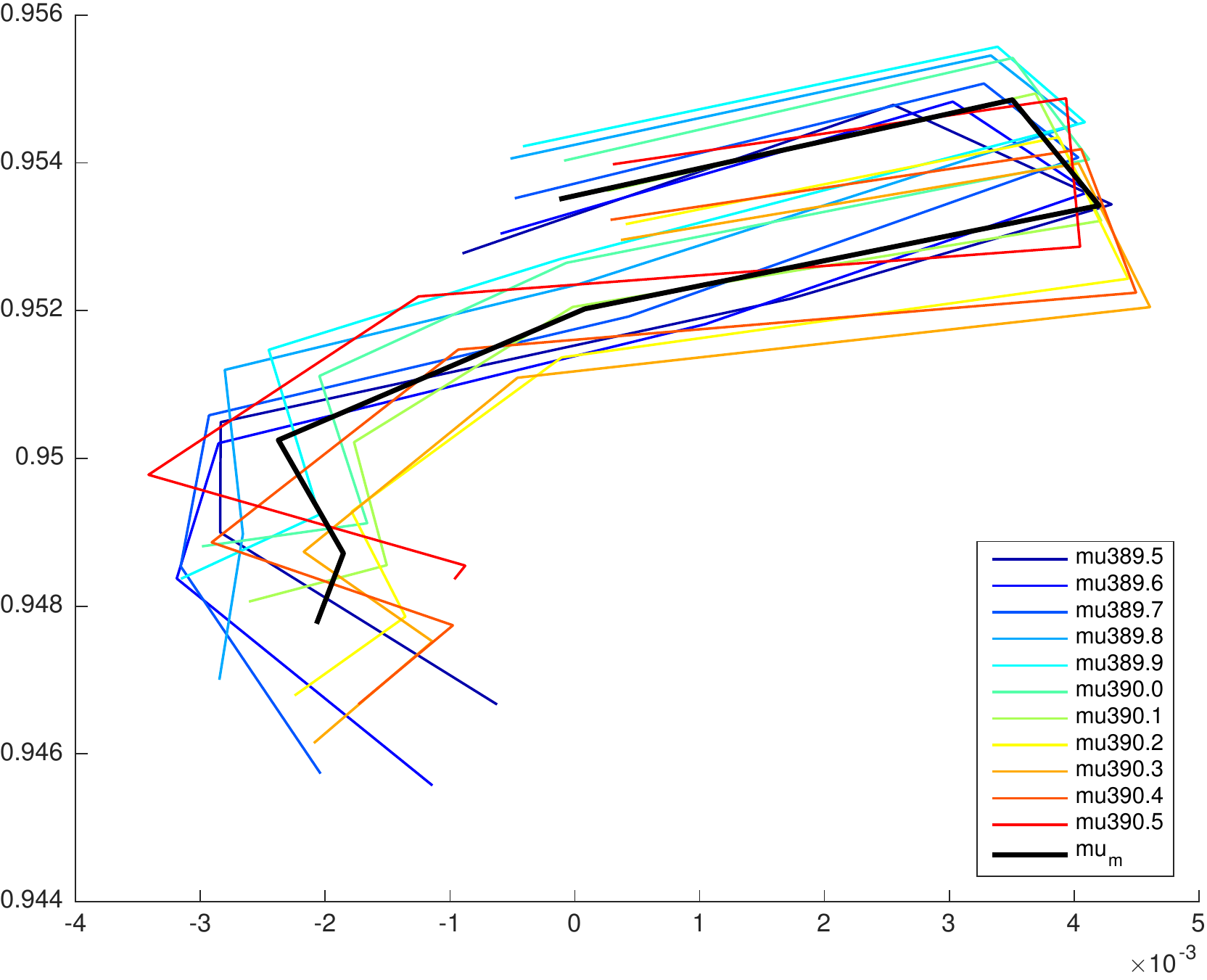}}
\caption{Computation of the mean curve (in black) for 4 sets of 11 curves in the hyperbolic half-plane, constructed from simulated helicopter radar data.}
\label{fig:karcherhelico}
\end{figure}

With this choice of representation, comparing two vectors of radar observations can be carried out by computing the distance between the two corresponding curves in the product manifold $\mathbb R_+^* \times \mathbb D$, which, thanks to the product metric, is the same as comparing their components separately -- that is, pairwise comparing the evolutions of each reflection coefficient $\mu_k(t)$ in the Poincar\'e disk. More generally, the representation of a vector of radar observation in a Riemannian manifold enables to do basic statistics on these objects, such as defining the mean, median and variance of a set, or performing classification. This can naturally be useful in target detection as well as target recognition. Here we use the algorithms presented in the previous sections to compute the Fr\'echet mean $\bar \mu(t)$ of $p$ curves $\mu^{(1)}(t), \hdots, \mu^{(p)}(t)$. The Fr\'echet mean, also called intrinsic mean, is defined by 
\begin{equation*}
\bar \mu =\underset{\mu \in \mathcal M}{\text{argmin}} \sum_{j=1}^{p} d(\mu,\mu^{(j)})^2,
\end{equation*}
if $\mathcal M$ is the space of curves in $\mathbb D$ and $d$ the distance on $\mathcal M$. Since it is defined as a minimizer of a functional, this intrinsic mean can be found by a gradient descent type procedure, summarized as follows.
\begin{algorithm}[Mean of a set of curves] 
\label{alg:karcher}
\leavevmode\par \noindent
Input : $(\mu^{(j)}(t)$, $1\leq j\leq p)$. \\
Initialize $\bar \mu$. Repeat until convergence :
\begin{enumerate} 
\item For $j = 1, \hdots, p$, compute the geodesic $\gamma^{(j)}(s)$ linking $\bar \mu$ to $\mu^{(j)}$ using geodesic shooting (Algorithm \ref{alg:geodshoot}) and its initial tangent vector $\gamma_s^{(j)}(0)$.
\item Update the mean $\bar \mu$ in the direction of the sum of the initial speed vectors
\begin{equation*}
\bar \mu \leftarrow \exp^{\mathcal M}_{\bar \mu}\left(\frac{1}{p}\sum_{j=1}^p \gamma^{(j)}_s(0)\right),
\end{equation*}
using the exponential map (Algorithm \ref{alg:expmap}).
\end{enumerate}
Output : $\bar \mu(t)$.
\end{algorithm}
Using our radar data, we compute the Fr\'echet mean for $4$ sets of $11$ curves tracing the evolutions of one reflection coefficient of $11$ signals -- we choose to represent only one of the coefficients $\mu_k$, $1\leq k\leq n-1$. These signals are generated using the helicopter signature simulator, and correspond to the observation at $4$ different times of $11$ helicopters which differ only in their rotor rotation speeds. We consider small variations (less than $1\%$) around the mean value of $390$ RPM (rotations per minute), and show the obtained curves in Figure \ref{fig:karcherhelico}. Theses curves are shown in the hyperbolic half-plane representation, which is equivalent to the Poincar\'e disk in terms of geometry. In each case, the red extremity of the colormap corresponds to the helicopter with the highest rotation speed, the blue extremity to the lowest rotation speed, and the mean curve is shown in black. This can be used to construct a "reference signature" for a given type of helicopter, for target recognition purposes.

\section{Conclusion}

We have studied a first-order Sobolev metric $G$ on the space of manifold-valued curves and its induced geometry. The metric $G$ can be obtained as the pullback of a natural metric on the tangent bundle $\text{T}\mathcal{M}$ by the square root velocity function, and as such it is reparametrization invariant. The special role that $G$ gives to the starting points of the curves induces a fiber bundle structure over the manifold $M$ seen as the set of starting points of the curves, for which the projection is a Riemannian submersion. The geodesic distance induced by $G$ takes into account the distance between the origins of the curve in $M$ and the $L^2$-distance between the $SRV$ representations, without parallel transporting the computations to a unique tangent plane as in \cite{lb} and \cite{zhang}. This should allow us to take into account a greater amount of information on the geometry of the manifold $M$. Using the pullback form of $G$, explicit equations can be obtained for the geodesics, as well as for Jacobi fields, which allow us to construct the optimal deformation between two curves by geodesic shooting. Once we can compute geodesics in the space of curves, we can also compute the mean of a set of curves and conceivably more. We considered the case where the base manifold $M$ is the hyperbolic half-plane, whose geometry coincides with the Fisher geometry of gaussian densities, and tested the algorithms on simulated radar data for the spectral analysis of locally stationary gaussian radar signals. Future work will include applications on the sphere for the statistical analysis of large trajectories.

\section*{Acknowledgments}

 This research was supported by Thales Air Systems and the french MoD DGA.

\section*{Appendix}

\begin{proof}[Proof of Proposition \ref{prop:geodH}]
The geodesic $\gamma(t)=x(t)+iy(t)$ linking two points vertically aligned $z_0=x_0+iy_0$ and $z_1=x_0+iy_1$ is a vertical segment $\gamma(t)=iy(t)$. It verifies the geodesic equation $\nabla_{\dot\gamma(t)}\dot\gamma(t)=0$. Using the expression \eqref{cov} of the covariant derivative of a vector field in $\mathbb H$, this gives the equation $\ddot y y = \dot y^2$, which can be rewritten as $\frac{\ddot y}{\dot y}=\frac{\dot y}{y}$. Integrating twice, we find that $y(t)=y_0 e^{t\ln\frac{y_1}{y_0}}$.

Now if $z_1=x_1+iy_1$ with $x_1 \neq x_0$, the geodesic $\gamma$ is the image by a Moebius transformation $a \mapsto \frac{az+b}{cz+d}$ (with $ad-bc=1$) of a vertical line $\bar \gamma(t)=i \bar y(t)$, which gives
\begin{equation}
\label{eqgeodh}
x(t) = \dfrac{ bd + ac \bar y(t)^2}{ d^2 + c^2 \bar y(t)^2}, \quad 
y(t) = \dfrac{ \bar y(t)}{ d^2 + c^2 \bar y(t)^2}.
\end{equation}
We know that $\gamma$ describes a semi-circle $\Omega$ whose origin $x_\Omega$ is on the x-axis, and that one end of the vertical line $\bar \gamma$ is sent to the point $a/c$ and the other to the point $b/d$. This implies that the center of the semi-circle is half-way between the two $x_\Omega = \frac{ad+bc}{2cd}$, and that the radius is $R=\frac{1}{2cd}$. These two equations as well as the condition $ad-bc=1$ gives a system of equations for the coefficients $a, b, c$ and $d$, which, if we choose to set $d=1$, yields the desired expressions. If the extremity $z_0$ is sent by the inverse of the obtained Moebius transformation on $i\bar y_0$, and $z_1$ on $i\bar y_1$, then the segment corresponding to the portion of $\gamma$ linking $z_0$ to $z_1$ is $\bar \gamma(t) = \bar y_0 e^{t\ln\frac{\bar y_1}{\bar y_0}}$. Taking the derivative of \eqref{eqgeodh} in $t=0$ gives the logarithm map.
\end{proof}

\begin{proof}[Proof of Proposition \ref{prop:expH}]
The exponential map uses the same equations as the logarithm map with the difference that $u_0$ is known instead of $z_1$. The proof is very similar to the the proof of Proposition \ref{prop:geodH} and is not detailed here.
\end{proof}

\begin{proof}[Proof of Proposition \ref{prop:ptH}]
Parallel transporting a vector $u_0\in T_{\gamma(t_0)}\mathbb H$ along a curve $\gamma$ from $t_0$ to $t$ is gives a vector field $u$ satisfying $\nabla_{\dot\gamma(t)}u=0$ and $u(t_0)=u_0$. Using Equation \eqref{cov}, this can be rewritten $\dot u = A u$ where 
\begin{equation*}
A=\dfrac{1}{y}  \left( \begin{matrix}
\dot{y} & \dot{x} \\ -\dot{x} & \dot{y} \end{matrix} \right).
\end{equation*}
$A$ is of the form $aI+bK$ where $I$ is the identity matrix and $K=\left( \begin{smallmatrix} 0 & 1 \\ -1 & 0 \end{smallmatrix} \right)$ and so the solution is $u(t)=u(t_0)\exp \int_{t_0}^t A(\tau) \mathrm d\tau$, that is, $u(t)=u(t_0)\exp B(t)$ with
\begin{equation*}
B(t)=\left(
\begin{matrix}
\ln \frac{y(t)}{y(t_0)}  & \int_{t_0}^t \frac{\dot{x}(\tau)}{y(\tau)}\mathrm d\tau \\
-\int_{t_0}^t \frac{\dot{x}(\tau)}{y(\tau)} \, \text{\footnotesize{$d\tau$}} & \ln \frac{y(t)}{y(t_0)}
\end{matrix} \right).
\end{equation*}
The matrix $B(t)$ is diagonalizable and therefore its exponential can be written
\begin{equation*}
\exp B(t)= e^{a(t_0,t)}\left( \begin{matrix}
\cos{\theta(t_0,t)} & \sin{\theta(t_0,t)} \\
-\sin{\theta(t_0,t)} & \cos{\theta(t_0,t)} \end{matrix} \right),
\end{equation*}
where $a(t_0,t)=\ln \frac{y(t)}{y(t_0)}$ and $\theta(t_0,t)=\int_{t_0}^{t} \frac{\dot{x}(\tau)}{y(\tau)} \, \text{\footnotesize{$d\tau$}}$. This gives us the desired formula.
\end{proof}


\end{document}